\documentclass[12pt,reqno]{amsart}
\usepackage[T2A]{fontenc}
\usepackage[english]{babel}
\usepackage{amsaddr}
\usepackage{amsmath}
\usepackage{amssymb}
\usepackage{amsfonts}
\usepackage{epsfig}
\usepackage{srcltx}
\usepackage{subfigure}
\usepackage{cite}
\usepackage{float}
\usepackage{xcolor}
\usepackage{mathtools}
\usepackage[a4paper, mag=1000, includefoot, left=2cm, right=2cm, top=2cm, bottom=2cm, headsep=1cm, footskip=1cm]{geometry}
   \usepackage[unicode,colorlinks]{hyperref}
    \hypersetup{colorlinks=true, citecolor=blue, linkcolor=blue}

\newtheorem{Th}{Theorem}
\newtheorem{Lem}{Lemma}

\begin{document}

\thispagestyle{empty}

\title[Nonlinear resonance in systems with noise]{Nonlinear resonance in systems with decaying perturbations and noise}

\author[O.A. Sultanov]{Oskar A. Sultanov}

\address{
Chebyshev Laboratory, St. Petersburg State University, 14th Line V.O., 29, Saint Petersburg 199178 Russia;\\
Institute of Mathematics, Ufa Federal Research Centre, Russian Academy of Sciences, Chernyshevsky street, 112, Ufa 450008 Russia.}
\email{oasultanov@gmail.com}


\maketitle

{\small
\begin{quote}
\noindent{\bf Abstract.} 
The effect of multiplicative white noise on the resonance capture in non-isochronous systems with time-decaying pumping is investigated. It is assumed that the intensity of perturbations decays with time, and its frequency is asymptotically constant. The occurrence of attractive solutions with an amplitude close to the resonant value and a phase synchronized with the excitation are considered. The persistence of such a regime in a stochastically perturbed system is analyzed. By combining the averaging method and the construction of suitable stochastic Lyapunov functions, conditions are derived that guarantee the stochastic stability of the resonant modes on infinite or asymptotically large time intervals. The proposed theory is applied to the Duffing oscillator with decaying parametric excitation and noise.

 \medskip



\end{quote}
}
{\small

\section*{Introduction}
Resonance is a phenomenon that results in a significant increase in the energy of a system due to oscillatory perturbations. The highest amplitude is achieved when the frequency of perturbations is close to the system's  natural frequency. Such phenomena play a crucial role in various fields, including micro- and nano-electronics, particle accelerators, and celestial mechanics (see, for example,~\cite{RS16,MKSS18}). Note that in non-isochronous systems, the natural frequency varies with the amplitude. In this case, the resonant influence of perturbations with constant frequency are expressed as oscillations with the corresponding resonant amplitude values. Such dynamics in nonlinear systems with small perturbations is usually associated with the nonlinear resonance and is considered to be well studied~\cite{BVC79,SUZ88,MAD98}. The persistent amplification of the amplitude can occur when the frequency of the driving changes continuously with time. Such effects in systems with small chirped frequency perturbations are related to the autoresonance phenomenon and have been studied in many papers~\cite{LF09,LK08,OK19,AK20}. In this paper, the presence of a small parameter is not assumed. We consider time-decaying oscillatory perturbations with asymptotically constant frequency and study the occurrence of nonlinear resonance in the presence of white noise disturbances. In this case, the perturbed stochastic equations have the form of asymptotically autonomous system.

Asymptotically autonomous systems arise, for example, in the study of various multidimensional problems by reducing the dimensionality of models~\cite{CCT95,DS22}, in the asymptotic integration of strongly nonlinear nonautonomous systems~\cite{BG08,KF13}, in problems with time-dependent damping~\cite{SFR19,JM23} and vanishing control~\cite{AT18}.

Qualitative properties of deterministic asymptotically autonomous systems have been studied in many papers. Note that in some cases, decaying perturbations may not affect the autonomous dynamics~\cite{RB53,LM56,LDP74}. However, in general, the asymptotic and qualitative properties of solutions to perturbed and unperturbed systems can differ significantly (see, for example,~\cite{HT94}). Bifurcations in such systems have been discussed in~\cite{LRS02,KS05,MR08,OS22Non}. This paper discuss the influence of damped oscillatory perturbations with an asymptotically constant frequency on nonlinear systems in the plane. Note that such perturbations has been considered in several papers. In particular, asymptotic integration problems were discussed for linear systems in~\cite{HL75,MP85,PN07,BN10}. Bifurcation phenomena in nonlinear systems related to the stability of the equilibrium were discussed in~\cite{OS21DCDS,OS21JMS}. Nonlinear resonance effects and the emergence of near-periodic attractive states with resonant amplitude values were studied in~\cite{OS25DCDS}. In this paper, we study the stability of such resonant regimes with respect to multiplicative stochastic perturbations of white noise type. 
 
 It is well known that even weak stochastic perturbations can lead the trajectories of systems to exit from any bounded domain~\cite{FW98}. Stochastic perturbations of scalar autonomous systems with decaying intensity  were considered in~\cite{AGR09,ACR11,KT13}. Bifurcations near the equilibrium and a description of various asymptotic regimes for solutions of asymptotically autonomous Hamiltonian systems with noise were discussed in~\cite{OS22IJBC,OS23SIAM}. The combined effect of multiplicative noise and decaying chirped frequency perturbations were studied in~\cite{OS25CNSNS}. However, the influence of stochastic disturbances on nonlinear resonance phenomena in systems subject to decaying driving with asymptotically constant frequency has not been studied earlier. This is the subject of the present paper. 
 
The paper is organized as follows. Section~\ref{sec1} provides the statement of the problem and a motivating example. The formulation of the main results is presented in Section~\ref{sec2}, while the justification is contained in subsequent sections. In Section~\ref{sec3}, we construct an averaging transformation that simplify the perturbed system in the first asymptotic terms at infinity in time. Section~\ref{sec4} analyzes the deterministic truncated system, obtained from the averaged system by dropping the diffusion coefficients, and describes the conditions for phase locking and phase drifting regimes. Section~\ref{sec5} discusses the persistence of the phase locking regime in the full system by constructing stochastic Lyapunov functions. In Section~\ref{sex}, the proposed theory is applied to examples of asymptotically autonomous systems. The paper concludes with a brief discussion of the results obtained.

\section{Problem statement}\label{sec1}
 Consider a system of It\^{o} stochastic differential equations 
\begin{gather}\label{PS}
d\begin{pmatrix}
r \\ \varphi
\end{pmatrix}=
\left({\bf a}_0(r)+{\bf a}(r,\varphi,S(t),t)\right)dt + \varepsilon {\bf A}(r,\varphi,S(t),t)\,d{\bf w}(t), \quad t\geq \tau_0>0,
\end{gather}
where ${\bf w}(t)\equiv (w_1(t),w_2(t))^T$ is a two dimensional Wiener process on a probability space $(\Omega,\mathcal F,\mathbb P)$, ${\bf A}(r,\varphi,S,t)\equiv \{\alpha_{i,j}(r,\varphi,S,t)\}_{2\times 2}$ is a $2\times 2$ matrix, ${\bf a}(r,\varphi,S,t)\equiv (a_1(r,\varphi,S,t),a_2(r,\varphi,S,t))^T$ and ${\bf a}_0(r)\equiv (0,\nu(r))^T$ are vector functions. A positive parameter $\varepsilon>0$ is used to control the intensity of the noise. The functions $\nu(r)>0$, $a_i(r,\varphi,S,t)$ and $\alpha_{i,j}(r,\varphi,S,t)$, defined for all $|r|\leq \mathcal R={\hbox{\rm const}}$, $(\varphi,S)\in\mathbb R^2$, $t\geq \tau_0$, are infinitely differentiable, $2\pi$-periodic with respect to $\varphi$ and $S$, and do not depend on $\omega\in\Omega$. The functions ${\bf a}(r,\varphi,S(t),t)$ and ${\bf A}(r,\varphi,S(t),t)$ correspond to perturbations of the autonomous system 
\begin{gather}
	\label{US}
		\frac{d\hat r}{dt}=0, \quad \frac{d\hat \varphi}{dt}=\nu(\hat r),
\end{gather}
describing non-isochronous oscillations on the plane $(x_1,x_2)=(\hat r\cos\hat \varphi,-\hat r\sin\hat \varphi)$ with a constant amplitude $\hat r(t)\equiv \varrho_0$, $|\varrho_0|< \mathcal R$ and a natural frequency $\nu(\varrho_0)$. In this case, the solutions $r(t)$ and $\varphi(t)$ of system \eqref{PS} play the role of the amplitude and the phase of the stochastically perturbed oscillations. 

It is assumed that the frequency of perturbations is asymptotically constant: $S'(t)\sim s_0$ as $t\to\infty$ with $s_0={\hbox{\rm const}}>0$, and the intensity of perturbations decays with time: for each fixed $r$ and $\varphi$
\begin{gather*}
{\bf a}(r,\varphi,S(t),t)\to 0, \quad {\bf A}(r,\varphi,S(t),t)\to 0, \quad t\to\infty.
\end{gather*}
In this case, the perturbed system \eqref{PS} is asymptotically autonomous, and the unperturbed system \eqref{US} coincides with the limiting system. It was shown in~\cite{OS25DCDS} that in the absence of a stochastic part of the perturbations ${\bf A}(r,\varphi,S,t)\equiv 0$, the system admits stable resonant solutions with the phase synchronized with the excitation and the amplitude tending to a non-zero resonant value. The goal of present paper is to take into account the effects of stochastic perturbations on the resonant dynamics.

Let us specify the class of perturbations. We assume that 
\begin{gather}\label{fgas}\begin{split}
&{\bf a}(r,\varphi,S,t)\sim  \sum_{k=n}^\infty   {\bf a}_k(r,\varphi,S) \mu^k(t), \\ 
&{\bf A}(r,\varphi,S,t)\sim  \sum_{k=p}^\infty  {\bf A}_k(r,\varphi,S) \mu^k(t), \quad 
t\to\infty,
\end{split}
\end{gather}
for all $|r|< \mathcal R$ and $(\varphi,S)\in\mathbb R^2$, where $n,p\in\mathbb Z_+=\{1,2,\dots\}$, $\mu(t)>0$ is strictly decreasing smooth function, defined for all $t\geq \tau_0$, and satisfies the following conditions:  
\begin{gather}\label{mucond}\begin{split}
&\mu(t)\to 0, \quad  \ell(t):=\frac{1}{\mu(t)}\frac{d\mu(t)}{dt} \to 0, \\ 
& \exists\, m\in\mathbb Z_+, \chi_m\leq 0: \quad  \frac{\ell(t)}{\mu^{m}(t)}\to \chi_m, \quad \frac{\mu^{m+1}(t)}{\ell(t)}\to 0, 
\end{split}
\end{gather}
as $t\to\infty$. Note that the function $t^{-1/q}$ with $q\in\mathbb Z_+$ is suitable for the role of $\mu(t)$. In this case, $\ell(t)\equiv -(1/q) t^{-1}$ and there exist $m=q$ and $\chi_m=-1/q$ such that the condition \eqref{mucond} holds. Another example is given by $\mu(t)\equiv t^{-1/q} \log t$ with $q\in\mathbb Z_+$. In this case, $\ell(t)\equiv - t^{-1}(1-q/\log t)/q$, $m=q$ and $\chi_m=0$.

The coefficients ${\bf a}_k(r,\varphi,S)\equiv (a_{1,k}(r,\varphi,S),a_{2,k}(r,\varphi,S))^T$ and ${\bf A}_k(r,\varphi,S)\equiv \{\alpha_{i,j,k}(r,\varphi,S)\}_{2\times 2}$ are $2\pi$-periodic with respect to $\varphi$ and $S$. 
The phase of perturbations is considered in the form
\begin{gather}\label{Sform}
S(t)=s_0 t+\sum_{k=1}^{d} s_k \int\limits_{t_0}^t \mu^k(\varsigma)\, d\varsigma,
\end{gather}
with some $d\in\mathbb Z_+$, where $s_k={\hbox{\rm const}}$, $s_0\neq 0$. Moreover, it is assumed that there exist $0<|r_0|< \mathcal R$ and coprime integers $\kappa,\varkappa\in\mathbb Z_+$ such that the resonant condition holds:
\begin{gather}\label{rc}
	\kappa s_0=\varkappa\nu(r_0), \quad \eta:=\nu'(r_0)\neq 0.
\end{gather}
The series in \eqref{fgas} are asymptotic as $t\to\infty$, and for all $N\geq 1$ the following estimates hold: 
\begin{align*}
 |{\bf a}(r,\varphi,S,t)-\sum_{k=n}^{n+N-1}  {\bf a}_k(r,\varphi,S) \mu^k(t)|=\mathcal O(\mu^{n+N}(t)), \\ 
 \| {\bf A}(r,\varphi,S,t)-\sum_{k=p}^{p+N-1}  {\bf A}_k(r,\varphi,S)\mu^k(t)\|=\mathcal O(\mu^{p+N}(t))
\end{align*} 
as $t\to\infty$ uniformly for all $|r|\leq \mathcal R$ and $(\varphi,S)\in\mathbb R^2$, where $\|\cdot\|$ is the operator norm coordinated with the norm $|\cdot|$ of $\mathbb R^2$

Consider the example
\begin{gather}\label{Ex0}
dx_1=x_2\, dt, \quad 
dx_2=\left(-x_1+\vartheta x_1^3+ t^{-\frac{n}{4}} \mathcal G(x_1,x_2,S(t)\right)\,dt+\varepsilon t^{-\frac{p}{4}} \mathcal B(S(t))\,dw_1(t),
\end{gather}
where 
\begin{gather*}
\mathcal G(x_1,x_2,S)\equiv  \mathcal P(S)x_1+ \mathcal Q(S)x_2, \quad S(t)\equiv \frac{3 t}{2}, \\
\mathcal P(S)\equiv \mathcal P_0+ \mathcal P_1 \sin S, \quad
\mathcal Q(S)\equiv \mathcal Q_0+ \mathcal Q_1 \sin S, \quad
\mathcal B(S)\equiv \mathcal B_0+ \mathcal B_1 \sin S,
\end{gather*}
with parameters $n,p\in\mathbb Z_+$, $\mathcal P_i$, $\mathcal Q_i$, $\mathcal B_i$, $\varepsilon\in\mathbb R$ and $\vartheta>0$. Let us show that in the new variables system \eqref{Ex0} takes the form \eqref{PS}. Consider the limiting system 
\begin{gather*}
\frac{d\hat x_1}{dt}=\hat x_2, \quad \frac{d\hat x_2}{dt}=- U'(\hat x_1), \quad U(x)\equiv \frac{x^2}{2}-\frac{\vartheta x^4}{4}.
\end{gather*} 
Note that the level lines $\{(x_1,x_2)\in\mathbb R^2: 2U(x_1)+x_2^2=r^2\}$ for all $0<|r|< (2\vartheta)^{-1/2}$ correspond to $T(r)$-periodic solutions
\begin{align}
\nonumber 
& \hat x_{1,0}(t,r)\equiv r\,{\hbox{\rm sn}}\left(\frac{t}{\sqrt{k_r^2+1}},k_r\right)\sqrt{k_r^2+1}, \\ 
\nonumber  & \hat x_{2,0}(t,r)\equiv r\,{\hbox{\rm cn}}\left(\frac{t}{\sqrt{k_r^2+1}},k_r\right){\hbox{\rm dn}}\left(\frac{t}{\sqrt{k_r^2+1}},k_r\right), \\
\label{omegaeq} 
& T(r)\equiv 4\mathcal K(k_r)\sqrt{k_r^2+1}, \quad 
\nu(r)\equiv \frac{2\pi}{T(r)},
\end{align}
where ${\hbox{\rm sn}}(u,k)$, ${\hbox{\rm cn}}(u,k)$, ${\hbox{\rm dn}}(u,k)$ are the Jacoby elliptic functions, $\mathcal K(k)$ is the complete elliptic integral of the first kind, and $k_r\in (0,1)$ is a root of the equation $(k_r+k_r^{-1})^{-2}=\vartheta  r^2/2$. Define auxiliary $2\pi$-periodic functions 
\begin{gather}\label{tr1}
X_1(\varphi,r)\equiv \hat x_{1,0} \left(\frac{\varphi}{\nu(r)},r\right), \quad
X_2(\varphi,r)\equiv \hat x_{2,0} \left(\frac{\varphi}{\nu(r)},r\right).
\end{gather}
It can easily be checked that 
\begin{gather*}
\nu(r)\partial_\varphi X_1=X_2, \quad \nu(r)\partial_\varphi X_2=-U(X_1),\\ 
2U(X_1)+X_2^2\equiv r^2, \quad
 {\hbox{\rm det}}\frac{\partial (X_1,X_2)}{\partial (\varphi,r)}\equiv \begin{vmatrix} \partial_\varphi X_1 & \partial_\varphi X_2\\ \partial_r X_1 & \partial_r X_2\end{vmatrix} \equiv \frac{r}{\nu(r)}.
\end{gather*}
Hence, the transformation \eqref{tr1} is invertible for all $0<|r|<(2\vartheta)^{-1/2}$ and $\varphi\in[0,2\pi)$. Then, by applying It\^{o}'s formula~\cite[\S 4.2]{BO98}, it can be shown that in the variables $(r,\varphi)$ system \eqref{Ex0}  takes the form \eqref{PS} with $\mu(t)\equiv t^{-1/4}$, $s_0=3/2$, $s_i=0$, 
\begin{gather}\label{aAEx0} 
\begin{split}
 a_i(r,\varphi,S,t)&\equiv t^{-\frac{n}{4}}a_{i,n}(r,\varphi,S)+ t^{-\frac{p}{2}} a_{i,2p}^\varepsilon(r,\varphi,S), \quad i\in\{1,2\}, \\
  \alpha_{i,j}(r,\varphi,S,t)&\equiv t^{-\frac{p}{4}}\alpha_{i,j,p}(r,\varphi,S),  \quad i,j\in\{1,2\},
\end{split}
\end{gather}
where
\begin{align*}
 a_{1,n}(r,\varphi,S)&\equiv \frac{X_2(\varphi,r)}{r}\left(\mathcal Q(S)X_1(\varphi,r)+\mathcal P(S)X_2(\varphi,r)\right), \\ 
a_{1,2p}^\varepsilon(r,\varphi,S)&\equiv   \frac{\varepsilon^2 \mathcal B^2(S) U(X_1(\varphi,r))}{  r^3}, \\
 a_{2,n}(r,\varphi,S)&\equiv - \frac{\nu(r)\partial_r X_1(\varphi,r)}{r}\left(\mathcal Q(S)X_1(\varphi,r)+\mathcal P(S)X_2(\varphi,r)\right), \\
a_{2,2p}^\varepsilon(r,\varphi,S)&\equiv  -  \frac{\varepsilon^2 \mathcal B^2(S) \nu(r)}{2 r}  \partial_r\left(\frac{\partial_rX_1(\varphi,r) X_2(\varphi,r)}{r}\right), \\
\alpha_{1,1,p}(r,\varphi,S)&\equiv  \frac{\mathcal B(S)X_2(\varphi,r)}{r}, \\ 
\alpha_{2,1,p}(r,\varphi,S)&\equiv  - \frac{\mathcal B(S)\nu(r)\partial_r X_1(\varphi,r)}{r},
\end{align*}
and $\alpha_{1,2,p}(r,\varphi,S)\equiv \alpha_{2,2,p}(r,\varphi,S)\equiv 0$. 
Note that $0<\nu(r)<1$ for all $0<|r|<(2\vartheta)^{-1/2}$. Hence, there exist $\kappa$, $\varkappa\in\mathbb Z_+$ and $0<|r_0|<(2\vartheta)^{-1/2}$ such that the condition \eqref{rc} holds. 

If $\mathcal G(x,y,S)\equiv0$ and $\varepsilon=0$, then $r(t)\equiv \varrho_0$ and $\varphi(t)\equiv \nu(\varrho_0)t+\phi_0$ with arbitrary constants $\varrho_0$ and $\phi_0$. 
Numerical analysis of the system with $\mathcal G(x,y,S)\not\equiv0$ and $\varepsilon=0$ shows that, depending on the values of the perturbation parameters, two cases are possible. 
In the first case, the amplitude of the solutions either tends to zero or grows until it reaches the boundary of the considered domain, depending on the sign of $\mathcal Q_0$ (see~Fig.~\ref{FigEx0}, a). In the second case, the solutions with an asymptotically constant amplitude $r(t)\approx r_0$ may appear (see~Fig.~\ref{FigEx0}, b). The emergence of such solutions is associated with the capture of the system into resonance. If $\varepsilon\neq 0$, the stochastic perturbations may disrupt such a behaviour (see~Fig.~\ref{FigEx0}, c). 

Thus, the goal of paper is to describe the conditions that guarantee the existence and stability of the resonant regimes in the perturbed stochastic systems of the form \eqref{PS} with perturbations satisfying \eqref{fgas},\eqref{mucond} and \eqref{Sform}.
\begin{figure}
\centering
\subfigure[$\mathcal P_1=0.4$, $\varepsilon=0$]{
\includegraphics[width=0.3\linewidth]{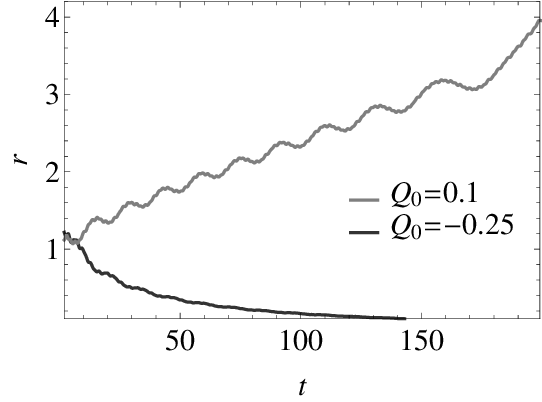}
}
\hspace{1ex}
 \subfigure[$\mathcal P_1=1$, $\varepsilon=0$]{
 \includegraphics[width=0.3\linewidth]{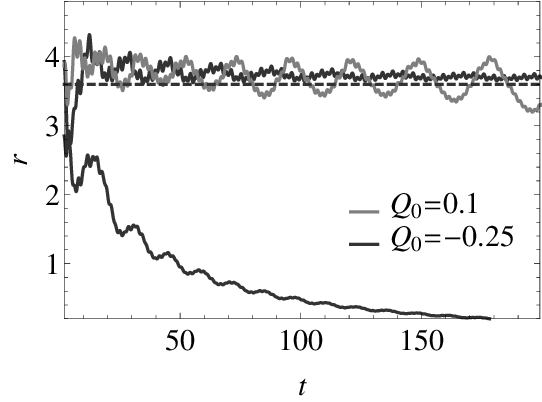}
}
\hspace{1ex}
\subfigure[$\mathcal P_1=1$, $\varepsilon=0.1$]{
 \includegraphics[width=0.3\linewidth]{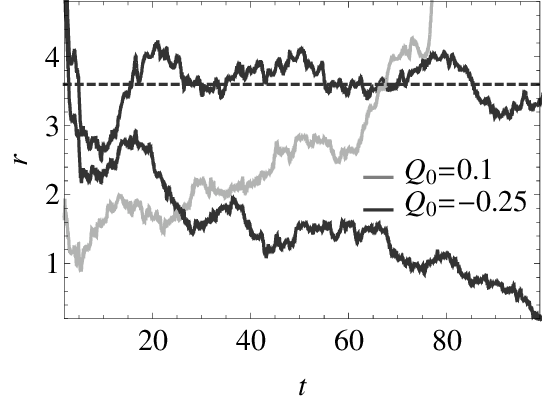}
}
\caption{\small 
The evolution of $r(t)\equiv \sqrt{2U(x(t))+y^2(t)}$ for sample paths of solutions to system \eqref{Ex0} with $s_0=3/2$, $n=2$, $p=1$, $\vartheta=2^{-5}$, $\mathcal P_0=\mathcal  Q_1=0$, $\mathcal B_0=3.6$, $\mathcal B_1=0$
and different values of the parameters $\mathcal Q_0$, $\mathcal P_1$, $\varepsilon$ and initial data. The dashed curves correspond to $r(t)\equiv 3.6$.} \label{FigEx0}
\end{figure}

\section{Main results}
\label{sec2}

Define the domain
\begin{align*}
&\mathfrak D_{\epsilon,\tau}:=\left\{(R,\Psi)\in\mathbb R^2: \quad \left|R+r_0  \mu^{-\frac{1}{2}}(\tau) \right|\leq \mathcal R \mu^{-\frac{1}{2}}(\tau)-\epsilon\right \}
\end{align*} 
with some $\epsilon\in [0, \mathcal R)$ and $\tau\geq \tau_0$. Denote by $\langle F(S)\rangle_{\varkappa S}$ the averaging of any function $F(S)$ over $S$ for the period $2\pi\varkappa$,
\begin{gather*}
\langle F(S)\rangle_{\varkappa S}\equiv \frac{1}{2\pi\varkappa}\int\limits_0^{2\pi\varkappa} F(S)\,dS.
\end{gather*}
Then, we have the following:
\begin{Th}\label{Th1}
Let system \eqref{PS} satisfy \eqref{fgas}, \eqref{Sform} and \eqref{rc} with $m\geq n$. Then, for all $2n-1\leq N\leq 2m$ and $\epsilon\in(0,\mathcal R)$ there exist $t_0\geq \tau_0$ and the transformations $(r,\varphi)\mapsto (R,\Psi)\mapsto (\rho,\psi)$,
\begin{gather}
\label{ch1} r(t)=r_0+R(t) \sqrt{\mu(t)} , \quad \varphi(t)=\frac{\kappa}{\varkappa}S(t)+\Psi(t), \\
\label{ch2}
 R(t)=\rho(t)+\tilde u_N(\rho(t),\psi(t),t), \quad \Psi(t)=\psi(t)+\tilde v_N(\rho(t),\psi(t),t),
\end{gather}
where $\tilde u_N(\rho,\psi,t)=\mathcal O(\mu^{\frac{1}{2}}(t))$, $\tilde v_N(\rho,\psi,t)=\mathcal O(\mu^{\frac{1}{2}}(t))$ as $t\to\infty$ uniformly for all $|\rho|<\infty$, $\psi\in\mathbb R$ and 
\begin{gather}\label{tildeest}
|\tilde u_N(\rho,\psi,t)|\leq \epsilon, \quad |\tilde v_N(\rho,\psi,t)|\leq \epsilon, \quad (\rho,\psi)\in\mathfrak D_{\epsilon,t_0}, \quad t\geq t_0,
\end{gather}
such that for all $0<|r|< \mathcal R$ and $\varphi\in\mathbb R$ system \eqref{PS} can be transformed into
\begin{gather}\label{rhopsi}
d\begin{pmatrix}
\rho \\ \psi
\end{pmatrix} = \begin{pmatrix} \Lambda(\rho,\psi,S(t),t) \\ 
\Omega(\rho,\psi,S(t),t)\end{pmatrix}dt+\varepsilon {\bf C}(\rho,\psi,S(t),t)\, d{\bf w}(t),
\end{gather}
with 
\begin{gather*}
\Lambda(\rho,\psi,S,t)  \equiv \hat\Lambda_N(\rho,\psi,t)+ \tilde\Lambda_N(\rho,\psi,S,t),\\ 
\Omega(\rho,\psi,S,t) \equiv \hat\Omega_N(\rho,\psi,t)+\tilde\Omega_N(\rho,\psi,S,t),
\end{gather*} 
and ${\bf C}(\rho,\psi,S,t)\equiv \{\sigma_{i,j}(\rho,\psi,S,t)\}_{2\times 2}$, defined for all $(\rho,\psi)\in\mathfrak D_{\epsilon,t_0}$, $t\geq t_0$ and $S\in\mathbb R$, such that
\begin{gather}
\nonumber
\hat\Lambda_N(\rho,\psi,t)\equiv \sum_{k=2n-1}^N \Lambda_k(\rho,\psi) \mu^{\frac{k}{2}}(t)-\frac{\ell(t)\rho}{2}, \quad 
\hat\Omega_N(\rho,\psi,t)\equiv \sum_{k=1}^N \Omega_k(\rho,\psi) \mu^\frac{k}{2}(t),\\
\label{tildeLO}\tilde\Lambda_N(\rho,\psi,S,t)=\mathcal O\left(\mu^{\frac{N+1}{2}}(t)\right), \quad
\tilde\Omega_N(\rho,\psi,S,t)=\mathcal O\left(\mu^{\frac{N+1}{2}}(t)\right), \\
\label{sigmaij}
\sigma_{1,j}(\rho,\psi,S,t)=\mathcal O\left(\mu^{\frac{2p-1}{2}}(t)\right), \quad
\sigma_{2,j}(\rho,\psi,S,t)=\mathcal O\left(\mu^{p}(t)\right)
\end{gather}
as $t\to\infty$ uniformly for all $|\rho|<\infty$ and $(\psi,S)\in\mathbb R^2$, where $\Lambda_k(\rho,\psi)$ and $\Omega_k(\rho,\psi)$ are $2\pi$-periodic in $\psi$, and are polynomials in $\rho$ of degree $k-1$ and $k$, respectively. 
In particular, $\Lambda_{2n-1}(\rho,\psi)\equiv \lambda(\psi):= \langle a_{1,n}(r_0,\kappa S/\varkappa+\psi,S)\rangle_{\varkappa S}$, $\partial_\psi\Omega_k(\rho,\psi)\equiv 0$ for $k\leq 2n-1$, and $\Omega_1(\rho,\psi)\equiv \eta\rho$.
\end{Th}

The proof is contained in Section~\ref{sec3}.

It can easily be checked that if $(\rho(t),\psi(t))\in\mathfrak D_{\epsilon,t_0}$ for all $t\geq t_0$, then, combining \eqref{ch1}, \eqref{ch2}, \eqref{tildeest} and the inequality $\mu(t)\leq \mu(t_0)$ as $t\geq t_0$, we obtain $|r(t)|\leq \mathcal R$ for all $t\geq t_0$. Thus, the solutions of system \eqref{rhopsi} with bounded $\rho(t)$ and $\psi(t)$ are associated with the resonant solutions of system \eqref{PS} such that $r(t)\approx r_0$ and $\varphi(t)\approx \kappa S(t)/\varkappa$.

Note that the transformation described in Theorem~\ref{Th1} averages only the drift terms of equations in the the leading asymptotic terms as $t\to\infty$, the diffusion coefficients are not simplified. The proposed method is based on the study of simplified averaged dynamics and the justification of its stochastic stability in the complete system.

Consider the truncated deterministic system of ordinary differential equations
\begin{gather}\label{trsys}
\frac{d\rho}{dt}=\Lambda(\rho,\psi,t), \quad \frac{d\psi}{dt}=\Omega(\rho,\psi,t)
\end{gather}
obtained from \eqref{rhopsi} by dropping the diffusion coefficients. The dynamics described by system \eqref{trsys} depends on the properties of the limiting system
\begin{gather}\label{limsys}
\frac{d\hat \rho}{dt} = \mu^{\frac{2n-1}{2}}(t)\lambda(\hat\psi), \quad \frac{d\hat \psi}{dt}=\mu^{\frac{1}{2}}(t) \eta \hat \rho.
\end{gather}
In particular, the presence and stability of the fixed points in system \eqref{limsys} play the important role. 
Consider two different cases:
\begin{align}
\label{aszero}
	\exists \, \psi_0\in\mathbb R: \quad & \lambda(\psi_0)=0, \quad \xi:=\partial_\psi \lambda(\psi_0)\neq 0; \\
\label{asnzero}  
	&\lambda(\psi)\neq 0 \quad \forall\, \psi\in\mathbb R.
\end{align}
In case of assumption \eqref{aszero}, it can easily be checked that the equilibrium $(0,\psi_0)$ of system \eqref{limsys} is of saddle type if  $\xi \eta>0$ and is of centre type if $\xi \eta<0$ (see Fig.~\ref{PhasePort}, a). In the first case, the discarded decaying terms of system \eqref{trsys} cannot qualitatively affect the dynamics~\cite{LM56}. However, in the latter case, the behavior of the trajectories depends significantly on the decaying terms of the equations. The condition $\xi=0$ corresponds to a center-saddle bifurcation in the limiting system. This degenerate case deserves special attention and is not considered in this paper. 

\begin{figure}
\centering
\subfigure[Case of \eqref{aszero}]{
\includegraphics[width=0.4\linewidth]{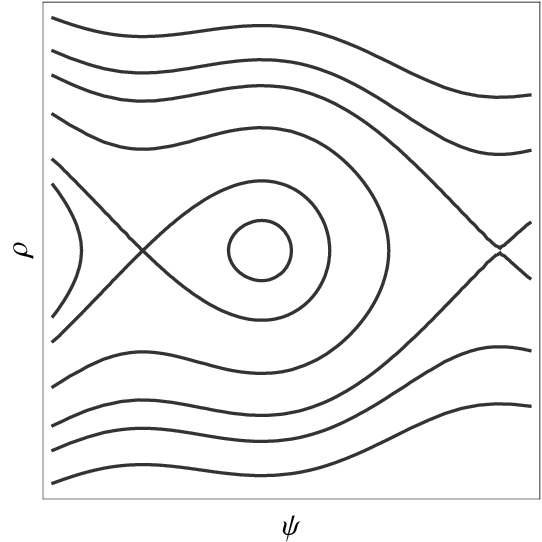}
}
\hspace{1ex}
 \subfigure[Case of \eqref{asnzero}]{
 \includegraphics[width=0.4\linewidth]{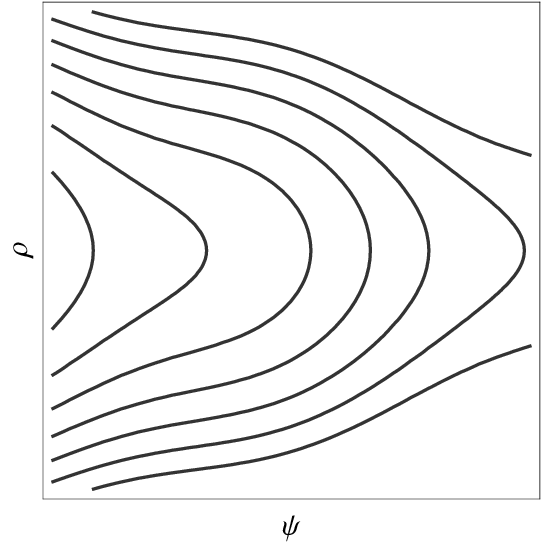}
}
\caption{\small A typical phase portrait of the limiting system \eqref{limsys} with $n=1$. } \label{PhasePort}
\end{figure}

Note that the model system \eqref{trsys} is hamiltonian in the leading asymptotic terms as $t\to\infty$. It follows from~\cite{OS22Non} that the qualitative properties of solutions in such asymptotically autonomous systems largely depend on non-Hamiltonian terms. With this in mind, we assume that
\begin{gather}\label{asst}
\begin{split}
\exists\,  h\in (1,N]:  \quad & \partial_\rho \Lambda_k(\rho,\psi)+\partial_\psi \Omega_k(\rho,\psi)\equiv 0, \quad k\leq h-1, \\
 & \gamma_h:= \partial_\rho \Lambda_h(0,\psi_0)+\partial_\psi \Omega_h(0,\psi_0)\neq 0.
\end{split}
\end{gather}
Note that the examples discussed in Section~\ref{sex} illustrate that this assumption is not too restrictive. The case when condition \eqref{asst} is not satisfied requires separate consideration and will be discussed elsewhere.

Define
\begin{gather}\label{zetah}
\zeta_h(t):=\int\limits_{t_0}^t \left(\mu(\varsigma)\right)^{\frac{h}{2}}\, d\varsigma, \quad \tilde \gamma_h:=\gamma_h-\delta_{h,2m} \frac{\chi_m n}{2}.
\end{gather}
Then, we have the following lemma.
\begin{Lem}\label{Lem2}
Let assumptions \eqref{aszero} and \eqref{asst} hold with $m\geq n$, $\xi\eta<0$ and $\tilde \gamma_h<0$. Then, system \eqref{trsys} has a stable particular solution $\rho_\ast(t)$, $\psi_\ast(t)$ such that
\begin{gather}\label{assol}
\rho_\ast(t)=\mathcal O(\mu^{z}(t)) , \quad \psi_\ast(t)=\psi_0+\mathcal O(\mu^z(t)), \quad t\to\infty
\end{gather}
with some $z\in (0,z_0)$, where
\begin{gather}\label{z0}
z_0=
	\begin{cases}
		\displaystyle \frac{1}{2}, & h<2m,\\
		\displaystyle \frac{1}{2}\min\left\{1,\frac{ \tilde \gamma_h }{ \chi_m } \right\}, & h=2m.
	\end{cases}
\end{gather}
  Moreover, the solution $\rho_\ast(t)$, $\psi_\ast(t)$ is asymptotically stable if $\zeta_h(t)\to\infty$ as $t\to\infty$.
\end{Lem}

In the case of assumption \eqref{asnzero}, it is easy to see that all trajectories of the limiting system \eqref{limsys} are unbounded (see Fig.~\ref{PhasePort}, b). The solutions of system \eqref{trsys} have a similar behavior. 

The proofs of Lemmas~\ref{Lem2} and \ref{Lem3} are contained in Section~\ref{sec4}.

\begin{Lem}\label{Lem3}
Let assumption \eqref{asnzero} hold with $m\geq n$ and $\zeta_{2n-1}(t)\to\infty$ as $t\to\infty$. Then, there exists $t_1\geq t_0$ such that $|\rho(t)|$ and $|\psi(t)|$ for solutions of system \eqref{trsys} with initial data $(\rho(t_1),\psi(t_1))\in\mathfrak D_{\epsilon,t_0}$ increase as $t>t_1$ until they reach the boundary of the domain $\mathfrak D_{\epsilon,t_0}$ in a finite time. 
\end{Lem}

This lemma describes the conditions for the phase drifting mode.
In this case, the phase $\varphi(t)$ of the system can significantly differ from the phase $\kappa S(t)/\varkappa$ of the perturbations and the solutions with $r(t)\approx r_0$ does not occur.

Let us show that the phase locking regime described by the solution $\rho_\ast(t)$, $\psi_\ast(t)$ of the truncated system \eqref{trsys} is preserved in the full stochastic system. Consider the function
\begin{gather*}
\mathcal M(\rho,\psi,t):= \sqrt{\left(\rho-\rho_\ast(t)\right)^2\mu^{-(n-1)}(t)+\left(\psi-\psi_\ast(t)\right)^2}
\end{gather*}
defined for all $(\rho,\psi)\in\mathbb R^2$ and $t\geq t_0$. 
We have the following theorem.
\begin{Th}\label{Th2}
Let system \eqref{PS} satisfy \eqref{fgas}, \eqref{Sform}, \eqref{rc}, and assumptions \eqref{aszero}, \eqref{asst} hold with $m\geq n$, $2p\geq n$ $\xi\eta<0$ and $\tilde\gamma_h <0$. Then, for all $\epsilon_1>0$, $\epsilon_2>0$, $l\in (0,1)$ and $t_\ast>t_0$ there exist $\delta_1>0$ and $\delta_2>0$ such that for any $0<\varepsilon\leq \delta_2$ the solution $\rho(t)$, $\psi(t)$ of system \eqref{rhopsi} with initial data
$ \mathcal M(\rho(t_\ast),\psi(t_\ast),t_\ast)\leq \delta_1$  satisfies
\begin{gather}\label{rpst}
\mathbb P\left(\sup_{0<t-t_\ast\leq \mathcal T} \mathcal M(\rho(t),\psi(t),t)\geq \epsilon_2\right)\leq \epsilon_1,
\end{gather}
where
\begin{gather*}
\mathcal T=
\begin{cases}
\infty, & \text{if} \quad \mu^{\frac{2p-n}{2}}(t)\in L_1(t_0,\infty),\\
\mathcal T_\varepsilon, & \text{if} \quad  \mu^{\frac{2p-n}{2}}(t)\not\in L_1(t_0,\infty)
\end{cases}
\end{gather*}
and $\mathcal T_\varepsilon>0$ is the root of the equation
\begin{gather}\label{Tequat}
\zeta_{2p-n}(\mathcal T_\varepsilon+t_\ast)-\zeta_{2p-n}(t_\ast)=\varepsilon^{-2(1-l)}.
\end{gather} 
\end{Th}
The proof is contained in Section~\ref{sec5}.

Note that estimates of the form \eqref{rpst} are associated with the stability in probability~\cite[\S 5.3]{RH12}, \cite[Ch. II, \S 1]{HJK67}. 

Thus, combining Theorem~\ref{Th1} and Theorem~\ref{Th2}, we see that there exists a stable phase locking regime in system \eqref{PS} with $r(t)\approx r_0$ and $\varphi(t)\approx \kappa S(t)/\varkappa+\psi_0$ with a high probability. In particular, depending on the decaying rate of the stochastic perturbations, the resonant regime is preserved over infinite or at least asymptotically large time intervals as $\varepsilon\to 0$.

\section{Change of variables}
\label{sec3}

\subsection{Amplitude residual and phase shift}

Substituting \eqref{ch1} into \eqref{PS} yields the following system:
\begin{gather}\label{RPsys}
d\begin{pmatrix}
R \\ \Psi
\end{pmatrix}=
{\bf b}(R,\Psi,S(t),t)\, dt + \varepsilon{\bf B}(R,\Psi,S(t),t)\,d{\bf w}(t),
\end{gather}
where ${\bf b}(r,\varphi,S,t)\equiv (b_1(r,\varphi,S,t),b_2(r,\varphi,S,t))^T$, ${\bf B}(R,\Psi,S,t)\equiv \{\beta_{i,j}(R,\Psi,S,t)\}_{2\times 2}$, 
\begin{align*}
b_1(R,\Psi,S,t) &\equiv
 \frac{1}{\sqrt{\mu(t)}} a_1\left(r_0+\sqrt{\mu(t)} R,\frac{\kappa}{\varkappa}S +\Psi,S,t\right)-\frac{\ell(t)R}{2}, \\
 b_2(R,\Psi,S,t) &
\equiv
\nu\left(r_0+\sqrt{\mu(t)} R\right)-\frac{\kappa}{\varkappa}S'(t)+a_2\left(r_0+\sqrt{\mu(t)} R,\frac{\kappa}{\varkappa}S +\Psi,S,t\right), \\
\beta_{1,j}(R,\Psi,S,t) & \equiv  \frac{1}{\sqrt{\mu(t)}} \alpha_{1,j}\left(r_0+\sqrt{\mu(t)} R,\frac{\kappa}{\varkappa}S +\Psi,S,t\right), \\
\beta_{2,j}(R,\Psi,S,t) &  \equiv \alpha_{2,j}\left(r_0+\sqrt{\mu(t)} R,\frac{\kappa}{\varkappa}S +\Psi,S,t\right).
\end{align*}
Using \eqref{fgas}, \eqref{Sform} and \eqref{rc}, we get
\begin{gather}\label{FGas}\begin{split}
{\bf b}(R,\Psi,S(t),t)&= \sum_{k=1}^{M}  {\bf b}_k(R,\Psi,S(t)) \mu^{\frac{k}{2}}(t)-\frac{\ell(t)}{2}\begin{pmatrix} R \\ 0 \end{pmatrix}+\tilde {\bf b}_{M}(R,\Psi,t), \\ 
{\bf B}(R,\Psi,S(t),t)&= \sum_{k=2p-1}^{2p-1+M} {\bf B}_k(R,\Psi,S(t))\mu^{\frac{k}{2}}(t)+\tilde {\bf B}_{2p-1+M}(R,\Psi,t), 
\end{split}
\end{gather}
where $\tilde {\bf b}_{M}(R,\Psi,t)=(\tilde b_{1,M}(R,\Psi,t),\tilde b_{2,M}(R,\Psi,t))^T$,
\begin{gather}\label{tildeas}
\begin{split}
& |\tilde {\bf b}_{M}(R,\Psi,t)|=\mathcal O\left(\mu^{\frac{M+1}{2}}(t)\right), \\ 
& \|\tilde {\bf B}_{2p-1+M}(R,\Psi,t)\|=\mathcal O\left(\mu^{\frac{2p+M}{2}}(t)\right)
\end{split}
\end{gather}
as  $t\to\infty$ uniformly for all $|R|<\infty$, $(\Psi,S)\in\mathbb R^2$ and for any integer $M\geq 1$. The coefficients ${\bf b}_k(R,\Psi,S)\equiv (b_{1,k}(R,\Psi,S),b_{2,k}(R,\Psi,S))^T$ and ${\bf B}_k(R,\Psi,S)\equiv \{\beta_{i,j,k}(R,\Psi,S)\}_{2\times 2}$ are $2\pi$-periodic in $\Psi$ and $2\pi \varkappa$-periodic in $S$.  In particular,
\begin{gather}\label{FkGk}
\begin{split}
& b_{1,k}(R,\Psi,S)\equiv \sum_{\substack{ i+2j=k+1 \\ i\geq n, j\geq 1}} \partial_r^i a_{1,j}\left(r_0,\frac{\kappa}{\varkappa}S +\Psi,S\right) \frac{R^i}{i!},\\
& b_{2,k}(R,\Psi,S)\equiv \partial_r^k\nu(r_0)\frac{R^k}{k!}-\frac{\kappa}{\varkappa} s_{k/2}+\sum_{\substack{ i+2j=k \\ i\geq n, j\geq 1}} \partial_r^i a_{2,j}\left(r_0,\frac{\kappa}{\varkappa}S +\Psi,S\right) \frac{R^i}{i!},\\
& \beta_{1,j,k}(R,\Psi,S)\equiv \sum_{\substack{ i+2l=k+1 \\ i\geq 0, l\geq p}} \partial_r^i \alpha_{1,j,k}\left(r_0,\frac{\kappa}{\varkappa}S +\Psi,S\right) \frac{R^i}{i!},\\
& \beta_{2,j,k}(R,\Psi,S)\equiv \sum_{\substack{ i+2l=k \\ i\geq 0, l\geq p}} \partial_r^i \alpha_{1,j,k}\left(r_0,\frac{\kappa}{\varkappa}S +\Psi,S\right) \frac{R^i}{i!}.
\end{split}
\end{gather}
We set $s_j=0$ for $j>d$ and $s_{k/2}=0$ for odd $k$. Note that $b_{1,j}(R,\Psi,S) \equiv \partial_\psi b_{2,j}(R,\Psi,S) \equiv 0$ for $j<2n-1$,
\begin{equation}\label{eq26}
\begin{array}{rcl}
b_{1,2n-1}(R,\Psi,S) & \equiv & \displaystyle a_{1,n}\left(r_0,\frac{\kappa}{\varkappa}S +\Psi,S\right), \\  
b_{2,1}(R,\Psi,S) & \equiv &\eta R, \\  
\beta_{1,j,2p-1}(R,\Psi,S)& \equiv &\displaystyle  \alpha_{1,j,p}\left(r_0,\frac{\kappa}{\varkappa}S +\Psi,S\right), \\ 
\beta_{2,j,2p-1}(R,\Psi,S)& \equiv& 0, \\
\beta_{2,j,2p}(R,\Psi,S)&\equiv &\displaystyle  \alpha_{2,j,p}\left(r_0,\frac{\kappa}{\varkappa}S +\Psi,S\right). 
\end{array}
\end{equation}

\subsection{ Near identity transformation}
Note that system \eqref{RPsys} is asymptotically autonomous with the limiting system 
\begin{gather*}
\frac{d\hat R}{dt}=0, \quad \frac{d\hat\Psi}{dt}=0, \quad \frac{d\hat S}{dt}=s_0.
\end{gather*}
In this case, the phase of the perturbations $S(t)$ can be considered as a fast variable as $t\to\infty$, and system \eqref{RPsys} can be simplified by averaging the drift terms of equations. Note that such method is effective in deterministic problems with a small parameter~\cite{BM61,AKN06}. The transformation is sought in the following form:
\begin{gather}\label{rpch}\begin{split}
	U_N(R,\Psi,S,t)&=R+\sum_{k=1}^N u_k(R,\Psi,S) \mu^{\frac{k}{2}}(t), \\
	V_N(R,\Psi,S,t)&=\Psi+\sum_{k=1}^N v_k(R,\Psi,S) \mu^{\frac{k}{2}}(t)
\end{split}
\end{gather}
with some integer $N\in [2n-1,2m]$. The coefficients $u_k(R,\Psi,S)$ and $v_k(R,\Psi,S)$ are assumed to be periodic with respect to $\Psi$ and $S$, and are chosen in such a way that the system written in the new variables 
$
\rho(t)\equiv U_N(R(t),\Psi(t),S(t),t)$, $
\psi(t)\equiv V_N(R(t),\Psi(t),S(t),t)
$
takes the form \eqref{rhopsi}, where the drift terms of equations do not depend explicitly on $S(t)$ at least in the first $N$ terms of the asymptotics as $t\to\infty$. Define the operators
\begin{align*}
&	\mathcal L U:=\partial_t U  + \left(\nabla_{(R,\Psi)} U\right)^T {\bf b}(R,\Psi,S(t),t)+\frac{\varepsilon^2}{2}{\hbox{\rm tr}}\left({\bf B}^T(R,\Psi,S(t),t) {\bf H}_{(R,\Psi)}(U){\bf B}(R,\Psi,S(t),t)\right), \\ 
& \nabla_{(R,\Psi)} U := \begin{pmatrix}\partial_{R} U  \\  \partial_{\Psi} U \end{pmatrix}, \\ 
&	{\bf H}_{(R,\Psi)}(U):= 	
\begin{pmatrix}
		\partial_{R}^2 U & \partial_{R}\partial_{\Psi} U \\
		\partial_{\Psi}\partial_{R} U & \partial_{\Psi}^2 U
\end{pmatrix}
\end{align*}
for any smooth function $U(R,\Psi,t)$. Note that $\mathcal L$ is the generator of the process defined by system \eqref{RPsys}. Then, by applying It\^{o}'s formula, we get
\begin{equation}\label{Itorhopsi}
\begin{array}{rcl}
d\rho &=&\displaystyle \mathcal L  U_N\, dt + \varepsilon(\nabla_{ (R,\Psi)} U_N)^T {\bf B}( R,\Psi,S(t),t)\, d{\bf w}, \\  
d\psi&=&\displaystyle \mathcal L  V_N\, dt + \varepsilon(\nabla_{ (R,\Psi)} V_N)^T {\bf B}( R,\Psi,S(t),t)\, d{\bf w}.
\end{array}
\end{equation}
Taking into account \eqref{Sform} and \eqref{FGas} with $M=N$, we see that the drift terms in \eqref{Itorhopsi} take the form
\begin{gather}\label{drpas}
\begin{split}
\mathcal L\begin{pmatrix} U_N  \\ V_N \end{pmatrix} = &
\sum_{k=1}^N \mu^{\frac{k}{2}}(t)\left\{{\bf b}_k+s_0\partial_S \begin{pmatrix} u_k \\ v_k \end{pmatrix}\right\} \\
& + 
\sum_{k=2}^N \mu^{\frac{k}{2}}(t) \sum_{j=1}^{k-1}\left\{ b_{1,j}\partial_R + b_{2,j}\partial_\Psi+s_{j/2} \partial_S \right\}\begin{pmatrix} u_{k-j} \\ v_{k-j} \end{pmatrix}\\
 & +\frac{\varepsilon^2}{2}\sum_{k=4p-1}^N \mu^{\frac{k}{2}}(t) \sum_{
        \substack{
             i+j+l=k\\
            i\geq 2p-1\\ j\geq 2p-1\\ l\geq 1
                }
            } 
\begin{pmatrix}
{\hbox{\rm tr}}\left \{{\bf B}^T_{i} {\bf H}_{(R,\Psi)}(u_{l} ){\bf B}_{j}\right\} \\ 
{\hbox{\rm tr}}\left\{{\bf B}^T_{i} {\bf H}_{(R,\Psi)}(v_{l} ){\bf B}_{j}\right\} \end{pmatrix} -\frac{\ell(t)}{2}\begin{pmatrix}R \\ 0\end{pmatrix} + \tilde{\bf Z}_N(R,\Psi,t),
\end{split}
\end{gather}
where  
\begin{gather}\label{zpas}\begin{split}
\tilde{\bf Z}_N(R,\Psi,t)\equiv & \tilde {\bf b}_N(R,\Psi,t)+
\sum_{k=N+1}^{2N} \mu^{\frac{k}{2}}(t) \sum_{j=1}^{k-1}\left\{ b_{1,j}\partial_R + b_{2,j}\partial_\Psi+s_{j/2} \partial_S \right\}\begin{pmatrix} u_{k-j} \\ v_{k-j} \end{pmatrix}\\
 & +\frac{\varepsilon^2}{2}\sum_{k= N+1}^{3N} \mu^{\frac{k}{2}}(t) \sum_{
        \substack{
            i+j+l=k\\
            i\geq 2p-1\\ j\geq 2p-1\\ l\geq 1
                }
            } 
\begin{pmatrix}
{\hbox{\rm tr}}\left ({\bf B}^T_{i} {\bf H}_{(R,\Psi)}(u_{l} ){\bf B}_{j}\right) \\ 
{\hbox{\rm tr}}\left({\bf B}^T_{i} {\bf H}_{(R,\Psi)}(v_{l} ){\bf B}_{j}\right) \end{pmatrix}\\
 & +\frac{\varepsilon^2}{2}\sum_{k= 1}^{N} \mu^{\frac{k}{2}}(t) 
\begin{pmatrix}
{\hbox{\rm tr}}\left (\tilde {\bf B}^T_{2p-1+N} {\bf H}_{(R,\Psi)}(u_{k} )\tilde{\bf B}_{2p-1+N}\right) \\ 
{\hbox{\rm tr}}\left(\tilde{\bf B}^T_{2p-1+N} {\bf H}_{(R,\Psi)}(v_{k} )\tilde{\bf B}_{2p-1+N}\right) \end{pmatrix} \\
& +\sum_{k=1}^N \mu^{\frac{k}{2}}(t) \left(\frac{\ell(t)k}{2}+\left(-\frac{\ell(t) R}{2}+\tilde b_{1,N}\right)\partial_R+\tilde b_{2,N}\partial_\Psi\right)\begin{pmatrix} u_{k} \\ v_{k} \end{pmatrix}.
\end{split}
\end{gather}
Combining \eqref{mucond}, \eqref{FGas} and \eqref{tildeas}, we get
\begin{gather}\label{tildeZas}
|\tilde {\bf Z}_{N}(R,\Psi,t)|=\mathcal O\left(\mu^{\frac{N+1}{2}}(t)\right)
\end{gather}
 as $t\to\infty$ uniformly for all $|R|<\infty$ and $\Psi\in\mathbb R$.
In \eqref{drpas} and \eqref{zpas} it is assumed that $u_k(R,\Psi,S)\equiv v_k(R,\Psi,S)\equiv 0$ for $k\leq 0$ and $k>N$. Comparing the drift terms in \eqref{rhopsi} with \eqref{drpas}, we obtain the following system for the coefficients $u_k(R,\Psi,S)$ and $v_k(R,\Psi,S)$:
\begin{gather}\label{ukvk}
s_0\partial_S \begin{pmatrix} u_k \\ v_k \end{pmatrix}=\begin{pmatrix} \Lambda_k(R,\Psi) \\ \Omega_k(R,\Psi)\end{pmatrix}-{\bf b}_k(R,\Psi,S)+\tilde{\bf b}_k(R,\Psi,S), \quad k=1,\dots, N,
\end{gather} 
where each $\tilde {\bf b}_k(R,\Psi,S)\equiv (\tilde b_{1,k}(R,\Psi,S),\tilde b_{2,k}(R,\Psi,S))^T$ is expressed in terms of $\{u_{j}, v_{j},\Lambda_{j}, \Omega_{j}\}_{j=1}^{k-1}$. In particular, if $p=n=1$, then
\begin{equation}\label{FGtilde}
\begin{array}{rl}
\tilde {\bf b}_1 \equiv 
& \begin{pmatrix}
0\\ 0
\end{pmatrix},\\
\tilde {\bf b}_2 \equiv 
& 
(u_1\partial_R+v_1\partial_\Psi)
\begin{pmatrix}
\Lambda_1\\ \Omega_1
\end{pmatrix}
-(b_{1,1}\partial_R + b_{2,1}\partial_\Psi )\begin{pmatrix} u_{1} \\ v_{1} \end{pmatrix}, \\
\tilde {\bf b}_3 \equiv
& \displaystyle
\sum_{i+j=3}(u_i\partial_R+v_i\partial_\Psi)
\begin{pmatrix}
\Lambda_j\\ \Omega_j
\end{pmatrix} + \frac{1}{2}\left(u_1^2\partial_R^2+2u_1v_1 \partial_R\partial_\Psi+v_1^2 \partial_\Psi^2\right)\begin{pmatrix}
\Lambda_1\\ \Omega_1
\end{pmatrix} \\
& \displaystyle -\sum_{j=1}^2 \left\{ b_{1,j}\partial_R + b_{2,j}\partial_\Psi+s_{j/2}\partial_S\right\}\begin{pmatrix} u_{3-j} \\ v_{3-j} \end{pmatrix} -\frac{\varepsilon^2}{2}\begin{pmatrix}
{\hbox{\rm tr}}\left ({\bf B}^T_{1} {\bf H}_{(R,\Psi)}(u_{1} ){\bf B}_{1}\right) \\ 
{\hbox{\rm tr}}\left({\bf B}^T_{1} {\bf H}_{(R,\Psi)}(v_{1} ){\bf B}_{1}\right) \end{pmatrix}.
\end{array}
\end{equation}
Let us define 
\begin{gather}\label{LambdaOmegak}
\begin{pmatrix}
\Lambda_k(R,\Psi) \\
\Omega_k(R,\Psi)
\end{pmatrix} \equiv \left\langle {\bf b}_k(R,\Psi,S)-\tilde {\bf b}_k(R,\Psi,S)\right\rangle_{\varkappa S}.
\end{gather}
Hence, system \eqref{ukvk} is solvable in the class of functions that are $2\pi\varkappa$-periodic in $S$ with zero mean. 
Note that the functions $u_k(R,\Psi,S)$, $v_k(R,\Psi,S)$, $\Lambda_k(R,\Psi)$, $\Omega_k(R,\Psi)$ are $2\pi$-periodic in $\Psi$. Moreover, it follows from \eqref{FkGk}, \eqref{FGtilde} and \eqref{LambdaOmegak}, that $u_k(R,\Psi,S)$, $\Lambda_k(R,\Psi)$ are polynomials in $R$ of degree $k-1$, and $v_k(R,\Psi,S)$, $\Omega_k(R,\Psi)$ are polynomials in $R$ of degree $k$. It can easily be checked that $\Lambda_k(R,\Psi)\equiv \partial_\Psi \Omega_k(R,\Psi)\equiv 0$ for $k< 2n-1$,  $\Lambda_{2n-1}(R,\Psi)\equiv\langle a_{1,n}(r_0,\kappa S/\varkappa+\Psi,S)\rangle_{\varkappa S}$ and $\Omega_1(R,\Psi)\equiv \eta R$. Combining this with \eqref{rpch}, we see that for all $\epsilon\in (0,\mathcal R)$ there exists $t_0\geq \tau_0$ such that
\begin{gather}\label{UVNest}
 |U_N(R,\Psi,S,t)-R|\leq \epsilon, \quad |V_N(R,\Psi,S,t)-\Psi|\leq \epsilon, \quad |{\hbox{\rm det}}\, {\bf J}_N(R,\Psi,S,t)-1|\leq \epsilon
\end{gather}
for all $(R,\Psi)\in\mathfrak D_{0,t_0}$, $t\geq t_0$ and $S\in\mathbb R$, where 
\begin{gather*}
{\bf J}_N(R,\Psi,S,t) := \begin{pmatrix}
        \partial_{R} U_N(R,\Psi,S,t) & \partial_\Psi U_N(R,\Psi,S,t)\\
        \partial_{R} V_N(R,\Psi,S,t) & \partial_\Psi V_N(R,\Psi,S,t)
    \end{pmatrix}.
\end{gather*}
Thus, \eqref{ch2} is invertible. Denote by $R= u(\rho, \psi, t)$, $\Psi=v(\rho,\psi,t)$ the inverse transformation defined for all $(\rho,\psi)\in\mathfrak D_{\epsilon,t_0}$ and $t\geq t_0$.
Then, 
\begin{align*}
\begin{pmatrix}
\tilde \Lambda_N(\rho,\psi,t)\\
\tilde \Omega_N(\rho,\psi,t)
\end{pmatrix}
\equiv & \,
\mathcal L
\begin{pmatrix}
U_N (R,\Psi,S(t),t)\\
V_N(R,\Psi,S(t),t)
\end{pmatrix}\Big|_{\substack{R=u(\rho, \psi, t)  \\ \Psi=v(\rho,\psi,t)}}  - \sum_{k=1}^N  \begin{pmatrix}
\Lambda_k(\rho,\psi)\\
\Omega_k(\rho,\psi)
\end{pmatrix} \mu^{\frac{k}{2}}(t), \\
{\bf C}(\rho,\psi,S(t),t)\equiv & \,
{\bf J}_N(R,\Psi,S(t),t)  {\bf B}(R,\Psi,S(t),t)\Big|_{\substack{ R=u(\rho, \psi, t) \\  \Psi=v(\rho,\psi,t)}}.
\end{align*}
Combining this with \eqref{drpas}, \eqref{tildeZas} and \eqref{eq26}, we get \eqref{tildeLO} and \eqref{sigmaij}. 
From \eqref{UVNest} it follows that $|\rho-u(\rho,\psi,t)|\leq \epsilon$ and $|\psi-v(\rho,\psi,t)|\leq \epsilon$ for all $(\rho,\psi)\in\mathfrak D_{\epsilon,t_0}$ and $t\geq t_0$. Define $\tilde u_N(\rho,\psi,t)\equiv  u(\rho,\psi,t)-\rho$, $\tilde v_N(\rho,\psi,t)\equiv  v(\rho,\psi,t)-\psi$. Thus, we have  \eqref{tildeest}. This completes the proof of Theorem~\ref{Th1}.

\section{Analysis of the truncated system}
\label{sec4}
\begin{proof}[Proof of Lemma~\ref{Lem2}]
Consider the following functions:
\begin{gather}\label{srp}
\varrho_\ast(t)=\sum_{k=1}^{h-1} \varrho_k \mu^\frac{k}{2}(t) , \quad 
\phi_\ast(t)=\psi_0+\sum_{k=1}^{h-1}\phi_k  \mu^\frac{k}{2}(t) ,
\end{gather}
where $\varrho_k$, $\phi_k$ are some constants. Substituting \eqref{srp} into \eqref{trsys} and equating the terms of like powers of $\mu(t)$ yield the chain of linear equations for the coefficients $\varrho_k$, $\phi_k$
\begin{gather}\label{rpk}
\begin{pmatrix}  
	\eta & 0 \\
 0 & \xi 
\end{pmatrix} 
\begin{pmatrix} 
 \varrho_k \\ \phi_k 
\end{pmatrix} 
= 
\begin{pmatrix} 
 \mathfrak  F_k \\ \mathfrak  G_k 
\end{pmatrix}, \quad k=1,\dots,h-1,
\end{gather}
where $\mathfrak  F_k$, $\mathfrak G_k$ are expressed through $\varrho_1,\phi_1, \dots, \varrho_{k-1}, \phi_{k-1}$. For instance,
\begin{align*}
\mathfrak  F_1=&-\Omega_2(0,\psi_0), \\ 
\mathfrak  G_1=&-\Lambda_{2n}(0,\psi_0),\\
\mathfrak  F_2=&-\Omega_{3}(0,\psi_0)-\left(\varrho_1\partial_\rho +\phi_1 \partial_\psi\right)\Omega_{2}(0,\psi_0), \\ 
\mathfrak  G_2=&-\Lambda_{2n+1}(0,\psi_0)-\sum_{i+j=3}\left(\varrho_i\partial_\rho +\phi_i \partial_\psi\right)\Lambda_{2n-2+j}(0,\psi_0) -\frac{1}{2}\left(\varrho_1^2\partial^2_\rho+2\varrho_1\phi_1\partial_\rho\partial_\psi+\phi_1^2\partial^2_\psi\right)\Lambda_{2n-1}(0,\psi_0).
\end{align*}
Since $\xi\eta\neq 0$, we see that system \eqref{rpk} is solvable. By construction, 
\begin{gather}
	\label{Zeq}
		\begin{split}
&\mathcal R_{\varrho}(t)\equiv \frac{d\varrho_{\ast}(t)}{dt}-\Lambda(\varrho_{\ast}(t),\phi_{\ast}(t),S(t),t)=\mathcal O\left(\mu^{\frac{h+2n-1}{2}}(t)\right),\\
&\mathcal R_{\phi}(t)\equiv \frac{d\phi_{\ast}(t)}{dt}-\Omega(\varrho_{\ast}(t),\phi_{\ast}(t),S(t),t)=\mathcal O\left(\mu^{\frac{h+1}{2}}(t)\right), \quad t\to\infty.
\end{split}
\end{gather}
Consider the change of variables
\begin{gather} \label{subsM}
\varrho(t)=\varrho_{\ast}(t)+\mu^{z+\frac{n-1}{2}}(t) u(t), \quad 
\phi(t)=\phi_{\ast}(t)+\mu^z(t) v(t).
\end{gather}
Substituting \eqref{subsM} into \eqref{trsys}, we obtain a perturbed near-Hamiltonian system
\begin{gather}\label{uvsys}
\frac{du}{dt}=-\partial_v H(u,v,t) + f(t), \quad 
\frac{dv}{dt}=\partial_u H(u,v,t)+G(u,v,t) + g(t), 
\end{gather}
with  
\begin{gather}\label{HGfg}\begin{split}
H(u,v,t)\equiv & \int\limits_0^u \mathcal G_\ast(w,0,t)\,dw - \int\limits_0^v \mathcal F_\ast(u,w,t)\,dw,\\ 
G(u,v,t)\equiv & \int\limits_0^v \left(\partial_u \mathcal F_\ast(u,w,t)+\partial_v \mathcal G_\ast(u,w,t)\right)\,dw,\\
f(t)\equiv & -\mu^{-z-\frac{n-1}{2}}(t) \mathcal R_\varrho(t),\\
g(t)\equiv & -\mu^{-z}(t)\mathcal R_\phi(t),
\end{split}
\end{gather}
where 
\begin{align*}
\mathcal  F_\ast(u,v,t)\equiv & \, \mu^{-z-\frac{n-1}{2}}(t)\left(\Lambda(\varrho_{\ast}(t)+ \mu^{z+\frac{n-1}{2}}(t) u,\phi_{\ast}(t)+\mu^{z}(t) v, S(t),t)-\Lambda(\varrho_\ast(t),\phi_\ast(t),S(t),t)\right) \\ & -\left(z+\frac{n-1}{2}\right) \ell(t) u, \\
\mathcal  G_\ast(u,v,t)\equiv &\, \mu^{-z}(t)\left(\Omega(\varrho_{\ast}(t)+ \mu^{z+\frac{n-1}{2}}(t) u,\phi_{\ast}(t)+\mu^{z}(t) v, S(t),t)-\Omega(\varrho_\ast(t),\phi_\ast(t),S(t),t)\right) \\
 & -z \ell(t) v.
\end{align*}
Define
\begin{gather}\label{ellchi}
\tilde \ell(t):=\frac{\ell(t)}{\mu^m(t)}-\chi_m, \quad \tilde \gamma_{h,z}:=\tilde\gamma_h-\delta_{2m,h}\chi_m 2z=\gamma_h-\delta_{2m,h}\chi_m \left(2z+\frac{n}{2}\right).
\end{gather}
Note that $\tilde \ell(t) \to 0$ as $t\to\infty$. By choosing $z\in(0,z_0)$, we can ensure that $\tilde \gamma_{h,z}<0$. It follows from  \eqref{aszero} and \eqref{Zeq} that
\begin{gather}\label{HU1}
\begin{split}
H(u,v,t) = &\frac{\mu^{\frac{n}{2}}(t)}{2}(\eta u^2-\xi v^2)\left\{1 +\mathcal O(\mu^{z}(t)) \right\},\\
 G(u,v,t)= & \mu^{\frac{h}{2}}(t) v\left\{\gamma_h+\mathcal O\left(\mu^z(t)\right)\right\} - \left(2z+\frac{n}{2}\right)\ell(t) v \\
 = &  \mu^{\frac{h}{2}}(t) v \left\{\tilde \gamma_{h,z} - \left(2z+\frac{n}{2}\right)\mu^{\frac{2m-h}{2}}(t) \tilde \ell(t) +\mathcal O\left(\mu^z(t)\right)\right\}\\
 = &  \mu^{\frac{h}{2}}(t) v \left\{\tilde \gamma_{h,z} +\mathcal O(\tilde \ell(t))+\mathcal O\left(\mu^z(t)\right)\right\},\\
 f(t)= & \mathcal O\left(\mu^{\frac{h+n-2z}{2}}(t)\right),\\ 
g(t)= & \mathcal O\left(\mu^{\frac{h+1-2z}{2}}(t)\right)
\end{split}
\end{gather}
as $\Delta=\sqrt{u^2+v^2}\to 0$ and $t\to\infty$. 

Consider
\begin{gather}\label{LFn}
\begin{split}
\mathcal V_0(u,v,t)\equiv &  \mu^{-\frac{n}{2}}(t)({\hbox{\rm sgn}}\, \eta) \left\{ H(u,v,t)  + \mu^{\frac{h}{2}}(t)\frac{\tilde \gamma_{h,z} u v}{2}\right\}
\end{split}
\end{gather}
as a Lyapunov function candidate for system \eqref{uvsys}. 
It follows from \eqref{HU1} that
\begin{gather*}
\mathcal V_0(u,v,t) = \frac{|\eta| u ^2 +  |\xi| v^2}{2}  + \mathcal O(\Delta^2)\mathcal O(\mu^z(t)), \quad \Delta\to 0, \quad t\to\infty.
\end{gather*}
Hence, there exist $\Delta_1>0$ and $t_1\geq  t_0 $ such that  
\begin{gather}\label{Lnineq1}
\mathcal V_- \Delta^2 \leq\mathcal V_0(u,v,t)\leq \mathcal V_+\Delta^2
\end{gather}
 for all $(u,v,t)\in\mathbb R^3$ such that $\Delta\leq \Delta_1$ and $t\geq t_1$, where $\mathcal V_+=\max\{|\xi|,|\eta|\}$ and $\mathcal V_-=\min\{|\xi|,|\eta|\}/4$. The derivative of $\mathcal V_0(u,v,t)$ with respect to $t$ along the trajectories of the system is given by 
\begin{gather*}
\frac{d \mathcal V_0}{dt}\equiv \mathcal Z_{1}(u,v,t)+\mathcal Z_{2}(u,v,t),
\end{gather*}
where $\mathcal Z_{1}\equiv \left(\partial_t -\partial_v H \partial_u +(\partial_u H +G) \partial_v \right)\mathcal V_0$ and $\mathcal Z_{2}\equiv \left(f \partial_u +g \partial_v \right)\mathcal V_0$. It can easily be checked that 
\begin{align*}
  \mathcal Z_{1}(u,v,t)& = \mu^{\frac{h}{2}}(t)(|\eta| u ^2 + |\xi| v^2) \left(\frac{\tilde \gamma_{h,z}}{2}+\mathcal O(\Delta)+\mathcal O(\tilde \ell(t))+\mathcal O(\mu^z(t))\right), \\
\mathcal  Z_{2}(u,v,t) & = \mathcal O(\Delta)\mathcal O\left(\mu^{\frac{h+1-2z}{2}}(t)\right)
\end{align*}
as $\Delta\to 0$ and $t\to\infty$. It follows that there exist $0<\Delta_2\leq \Delta_1$ and $t_2\geq t_1$ such that 
\begin{gather*}
\mathcal Z_{1}(u,v,t)\leq - K_1  \mu^{\frac{h}{2}}(t) \Delta^2, \quad 
\mathcal Z_{2}(u,v,t)\leq K_2 \mu^{\frac{h}{2}}(t) \varsigma(t) \Delta
\end{gather*}
for all $(u,v,t)\in\mathbb R^3$ such that $\Delta\leq \Delta_2$ and $t\geq t_2$, where $K_1=|\tilde \gamma_{h,z}| \mathcal V_->0$, $K_2>0$ and $ \varsigma(t)\equiv \mu^{(1-2z)/2}(t)$ is positive strictly decreasing function as $t\geq t_2$. Therefore, for all $\epsilon\in (0,\Delta_2)$ there exist $\delta_\epsilon= 2K_2 \varsigma(t_\epsilon)/K_1$ and $t_\epsilon\geq t_2$ such that $\varsigma(t_\epsilon)\leq \epsilon K_1/ (2K_2)$ and the following inequality holds:
\begin{gather*}
\frac{d\mathcal V_0}{dt} \leq 	\mu^{\frac{h}{2}}(t) \left(-K_1+\frac{K_2\varsigma(t_\epsilon)}{\delta_\epsilon}\right)\Delta^2 \leq 0
\end{gather*}
for all $(u,v,t)\in\mathbb R^3$ such that $\delta_\epsilon\leq \Delta\leq \epsilon$ and $t\geq t_\epsilon$. Combining this with \eqref{Lnineq1}, we see that any solution of system \eqref{uvsys} with initial data $\sqrt{u^2(t_\epsilon)+v^2(t_\epsilon)}\leq \delta$, where $\delta=\max\{\delta_\epsilon,\epsilon\sqrt{\mathcal V_-/\mathcal V_+}\}$, cannot exit from the domain $\{(u,v)\in\mathbb R^2: \Delta\leq \epsilon\}$ as $t\geq t_\epsilon$. It follows from \eqref{subsM} that the trajectories of system \eqref{trsys} starting close to $\varrho_\ast(t)$, $\phi_\ast(t)$ satisfy the estimates $\varrho(t)=\varrho_{\ast}(t)+\mathcal O(\mu^{z+\frac{n-1}{2}}(t))$, $\phi(t)=\phi_{\ast}(t)+\mathcal O(\mu^z(t))$ as $t\to\infty$. Thus, there exists a solution $\rho_\ast(t)$, $\psi_\ast(t)$ of system \eqref{trsys} with asymptotics \eqref{assol}. 

To prove the stability of the solution $\rho_\ast(t)$, $\psi_\ast(t)$, consider the change of variables 
\begin{gather} \label{subsast}
\varrho(t)=\rho_{\ast}(t)+ \mu^{\frac{n-1}{2}}(t) u(t), \quad \phi(t)=\psi_{\ast}(t)+ v(t).
\end{gather} 
Substituting \eqref{subsast} into \eqref{trsys}, we obtain system \eqref{uvsys}, where
\begin{gather}\label{FGast2}
\begin{split}
\mathcal  F_\ast(u,v,t)&\equiv \mu^{-\frac{n-1}{2}}(t)\Lambda(\rho_{\ast}(t)+ \mu^{\frac{n-1}{2}}(t) u,\psi_{\ast}(t)+v, S(t),t)-\Lambda(\rho_\ast(t),\psi_\ast(t),S(t),t)-\frac{n-1}{2}\ell(t)u, \\
\mathcal  G_\ast(u,v,t)&\equiv\Omega(\rho_{\ast}(t)+\mu^{\frac{n-1}{2}}(t)u,\psi_{\ast}(t)+v, S(t),t)-\Omega(\rho_\ast(t),\psi_\ast(t),S(t),t)
\end{split}
\end{gather}
and $f(t)\equiv g(t)\equiv 0$. In this case,
\begin{gather}\label{HU2}
\begin{split}
 H(u,v,t) = &\frac{\mu^{\frac{n}{2}}(t)}{2}(\eta u^2-\xi v^2)\left\{1 +\mathcal O(\Delta)+\mathcal O(\mu^{z}(t)) \right\},\\
 G(u,v,t)=   &  \mu^{\frac{h}{2}}(t) v \left\{\tilde \gamma_{h} +\mathcal O(\Delta)+\mathcal O(\tilde \ell(t))+\mathcal O\left(\mu^z(t)\right)\right\}
\end{split}
\end{gather}
as $\Delta=\sqrt{u^2+v^2}\to 0$ and $t\to\infty$. Then, by using Lyapunov function \eqref{LFn} and repeating the arguments as given above, we get 
\begin{gather*}
\frac{d\mathcal V_0}{dt}\leq - \mu^{\frac{h}{2}}(t) K_3 \mathcal V_0
\end{gather*}
 for all $(u,v,t)\in\mathbb R^3$ such that $\Delta\leq \Delta_3$, $t\geq t_3$ with some $0<\Delta_3\leq \Delta_1$, $t_3\geq t_1$ and $K_3=|\tilde \gamma_{h}| \mathcal V_-/\mathcal V_+>0$. Integrating this inequality and taking into account \eqref{Lnineq1}, we obtain
\begin{align}\label{inequv2}
&u^2(t)+v^2(t)\leq  (u^2(t_3)+v^2(t_3))  \exp \left\{ -K_3 (\zeta_h(t)-\zeta_h(t_3))\right\} \frac{\mathcal V_+ }{\mathcal V_-}, \quad t\geq t_3.
\end{align}
Therefore, for all $\epsilon\in (0,\Delta_3)$ there exists $\delta_\epsilon=\epsilon\sqrt{\mathcal V_-/(2\mathcal V_+)}$ such that any solution with initial data $u^2(t_3)+v^2(t_3)\leq \delta_\epsilon^2$ satisfy the estimate $u^2(t)+v^2(t)\leq \epsilon^2$ as $t>t_3$. Combining this with \eqref{subsast}, we see that the solution $\rho_\ast(t)$, $\psi_\ast(t)$ is stable. Moreover, it follows from \eqref{inequv2} that the solution is asymptotically stable if $\zeta_h(t)\to\infty$ as $t\to\infty$. 
\end{proof}

\begin{proof}[Proof of Lemma~\ref{Lem3}]
To be specific, let $r_0>0$. Define $D_\pm=(\mathcal R\pm r_0)\mu^{-1/2}(t_0)-\epsilon$.  
It follows from the first equation in \eqref{trsys} and assumption \eqref{asnzero} that there exist $t_1\geq t_0$ and $K_1>0$ such that 
$ |{d\rho}/{dt}|\geq \mu^{\frac{2n-1}{2}}(t)K_1$ for all $(\rho,\psi)\in\mathfrak D_{\epsilon,t_0}$ and $t\geq t_1$. Integrating this inequality yields 
\begin{gather*}
|\rho(t)-\rho(t_1)|\geq K_1 (\zeta_{2n-1}(t)-\zeta_{2n-1}(t_1))> 0, \quad t> t_1.
\end{gather*} 
Then, $|\rho(t)|$ with initial data $|\rho(t_1)|<D_-$ increases until it reaches the boundary $|\rho|=D_-$ in a finite time $t_e>t_1$. Moreover, for all initial data $|\rho(t_1)|\leq D_-/4$ and $\psi(t_1)\in\mathbb R$ there exists $t_2 \in[ t_1,t_e)$ such that $|\rho(t)|\geq D_-/2$ as $t\in [t_2,t_e]$. Combining this with the second equation in \eqref{trsys}, we see that there exists $K_2>0$ such that 
$|{d\psi}/{dt} |\geq \mu^{\frac{1}{2}}(t) K_2$
for all $D_-/2\leq |\rho|\leq D_-$, $\psi\in\mathbb R$ and $t\geq t_2$. Then, by integration, we have
 \begin{gather*}
	|\psi(t)-\psi(t_2)|\geq K_2 (\zeta_1(t)-\zeta_1(t_2)), \quad t\in [t_2,t_e].
\end{gather*}
Therefore, $|\psi(t)|$ increases for all $t\in [t_2,t_e]$.
Similarly, if $-D_+<\rho(t_1)<- D_-$, it can be shown that $|\rho(t)|$ increases until it reaches the boundary $\rho=-D_+$ or $\rho=D_-$ in a finite time.
\end{proof}

\section{Persistence of the phase locking}
\label{sec5}
Consider the change of variables \eqref{subsast}, where $\rho_\ast(t)$, $\psi_\ast(t)$ is the solution of the truncated system  \eqref{trsys}, described in Lemma~\ref{Lem2}. Substituting \eqref{subsast} into \eqref{rhopsi}, we obtain
\begin{gather}\label{uvsde}
d\begin{pmatrix}
u \\ v
\end{pmatrix} = \begin{pmatrix} -\partial_v H(u,v,t) \\ 
\partial_u H(u,v,t)+G(u,v,t)\end{pmatrix}dt+\varepsilon {\bf D}(u,v,t)\, d{\bf w}(t),
\end{gather}
where ${\bf D}(u,v,t)\equiv \{d_{i,j}(u,v,t)\}_{2\times 2}$, $d_{i,j}(u,v,t) \equiv\sigma_{i,j}(\rho_\ast(t)+\mu^{\frac{n-1}{2}}(t)u,\psi_\ast(t)+v,S(t),t)$, the functions $H(u,v,t)$, $G(u,v,t)$ are defined by \eqref{HGfg}, \eqref{FGast2} and satisfy asymptotic estimates \eqref{HU2}. Note that $\partial_{u} H(0,0,t)\equiv  \partial_{v} H(0,0,t)\equiv G(0,0,t)\equiv 0$. It follows from \eqref{sigmaij} that
\begin{gather*}
d_{1,j}(u,v,t)=\mathcal O\left(\mu^{\frac{2p-n}{2}}(t)\right), \quad
d_{2,j}(u,v,t)=\mathcal O\left(\mu^{p}(t)\right)
\end{gather*}
as $t\to\infty$ uniformly for all $(u,v)\in\mathbb R^2$ such that $\Delta=\sqrt{u^2+v^2}\leq \Delta_0$ with some $\Delta_0>0$. If $\varepsilon=0$, system \eqref{uvsde} has an equilibrium at the origin. Let us prove the stability of the equilibrium with respect to white noise by constructing a suitable stochastic Lyapunov function~\cite[Ch. II, \S 5]{HJK67}. 

The generator of the process defined by \eqref{uvsde} has the  form $\mathfrak L\equiv  \mathfrak L_0+\varepsilon^2 \mathfrak L_1$, where
\begin{align*}
\mathfrak L_0:=& \partial_t-\partial_v H \partial_u+ (\partial_u H+G)\partial_v,\\
\mathfrak L_1:=& \frac{1}{2}\left((d_{1,1}^2+d_{1,2}^2)\partial_u^2+2(d_{1,1}d_{2,1}+d_{1,2}d_{2,2})\partial_u\partial_v+(d_{2,1}^2+d_{2,2}^2)\partial_v^2\right).
\end{align*}
Consider the function $\mathcal V_0(u,v,t)$ defined by \eqref{LFn}. It can easily be checked that 
\begin{align*}
 \mathcal V_0(u,v,t) & =  \frac{1}{2}(|\eta| u ^2 + |\xi| v^2)+ \mathcal O(\Delta^3)+ \mathcal O(\Delta^2)\mathcal O(\mu^z(t)), \\
 \mathfrak L_0 \mathcal V_0(u,v,t) & = \mu^{\frac{h}{2}}(t)(|\eta| u ^2 + |\xi| v^2) \left(\frac{\tilde \gamma_h}{2}  +\mathcal O(\Delta) +\mathcal O(\tilde \ell(t))+ \mathcal O(\mu^z(t))\right), \\
 \mathfrak L_1 \mathcal V_0(u,v,t) & = \mathcal O\left(\mu^{\frac{2p-n}{2}}(t)\right)
\end{align*}
as $t\to\infty$ and $\Delta\to 0$ with some $z\in (0,z_0)$, where $z_0$ and $\tilde \ell(t)$ are defined by \eqref{z0} and \eqref{ellchi}. Hence, there exist 
$t_1\geq t_0$, 
$\Delta_\ast>0$, 
$\mathcal V_+=\max\{|\xi|,|\eta|\}>0$, 
$\mathcal V_-=\min\{|\xi|,|\eta|\}/4>0$
$K_1=|\tilde\gamma_h |\mathcal V_->0$ and $K_2>0$ such that
\begin{gather}\label{V0ests}
\begin{split}
 \mathcal V_- \Delta^2\leq \mathcal V_0(u,v,t) &\leq \mathcal V_+ \Delta^2, \\
\mathfrak L \mathcal V_0(u,v,t)& \leq - K_1 \mu^{\frac{h}{2}}(t) \Delta^2+ \varepsilon^2 K_2   \mu^{\frac{2p-n}{2}}(t)
\end{split}
\end{gather}
for all $t\geq t_1$ and $\Delta\leq \Delta_\ast$.
Fix the parameters $\epsilon_1\in (0,\Delta_\ast)$, $\epsilon_2\in (0,1)$ and $t_\ast\geq \tau_1$. Consider the stochastic Lyapunov function candidate for system \eqref{uvsde} in the form
\begin{gather}\label{SLF1}
\mathcal V(u,v,t)\equiv \mathcal V_0(u,v,t)+\varepsilon^2 \mathcal V_1(t,\mathcal T),   
\end{gather}
where $\mathcal V_1(t,\mathcal T)\equiv 
 K_2 \left(\zeta_{2p-n}(t_\ast+\mathcal T)-\zeta_{2p-n}(t)\right)$
with some parameter $\mathcal T>0$. It can easily be checked that 
\begin{gather}
\label{V1est}
\begin{split}
\mathcal V(u,v,t) & \geq \mathcal V_0(u,v,t)\geq 0, \\ 
\mathfrak L \mathcal V(u,v,t) & \leq - K_1 \mu^{\frac{h}{2}}(t) \Delta^2\leq 0
\end{split}
\end{gather}
 for all $(u,v,t)\in \mathcal I(\Delta_\ast,t_\ast,\mathcal T):=\{(u,v,t)\in\mathbb R^3: \Delta\leq \Delta_\ast, 0\leq t-t_\ast\leq \mathcal T\}$. Let ${\bf y}(t)=(u(t),v(t))^T$ be a solution of system \eqref{uvsde} with initial data $|{\bf y}(t_\ast)|\leq  \delta_1$ and $\theta_{\epsilon}$ be the first exit time of $({\bf y}(t), t)$ from the domain $\mathcal I(\epsilon,t_\ast,\mathcal T)$ with some $\delta_1\in (0,\epsilon_1)$. Define the function $\theta_{\epsilon_1}(t)\equiv \min\{\theta_{\epsilon_1}, t\}$. Then, ${\bf y}(\theta_{\epsilon_1}(t))$ is the process stopped at the first exit time from the domain $\mathcal I(\epsilon_1,t_\ast,\mathcal T)$. It follows from \eqref{V1est} that $\mathcal V(u(\theta_{\epsilon_1}(t)), v(\theta_{\epsilon_1}(t)),\theta_{\epsilon_1}(t))$ is a nonnegative supermartingale (see, for example,~\cite[\S 5.2]{RH12}). In this case, the following estimates hold:
\begin{align*}
\mathbb P\left(\sup_{0\leq t-t_\ast\leq \mathcal T} |{\bf y}(t)|\geq \epsilon_1\right) & = 
\mathbb P\left(\sup_{t\geq t_\ast} |{\bf y}(\theta_{\epsilon_1}(t))|\geq \epsilon_1\right)  \\
& \leq \mathbb P\left(\sup_{t\geq t_\ast} \mathcal V(u(\theta_{\epsilon_1}(t)), v(\theta_{\epsilon_1}(t)),\theta_{\epsilon_1}(t))\geq \mathcal V_-\epsilon_1^2\right) \\
&\leq 
\frac{ \mathcal V(u(\theta_{\epsilon_1}(t_\ast)), v(\theta_{\epsilon_1}(t_\ast)),\theta_{\epsilon_1}(t_\ast))}{\mathcal V_-\epsilon_1^2}.
\end{align*}
The last estimate follows from Doob's inequality for supermartingales. It follows from \eqref{zetah}, \eqref{V0ests} \eqref{SLF1} that 
\begin{gather*}
\mathcal V(u(\theta_{\epsilon_1}(t_\ast)), v(\theta_{\epsilon_1}(t_\ast)),\theta_{\epsilon_1}(t_\ast))\leq \mathcal V_+ \delta_1^2+\varepsilon^{2(1-l)} \delta_2^{2l} \mathcal V_1(t_\ast,\mathcal T)
\end{gather*}
 for all $0<\varepsilon\leq \delta_2$ with some  $l\in(0,1]$ and $\delta_2>0$.

Let $\mu^{(2p-n)/2}(t)\in L_1(t_0,\infty)$. Then, there exists $\mathcal V_\ast>0$ such that $\mathcal V_1(t_\ast,\mathcal T)\leq \mathcal V_\ast$ for all $\mathcal T>0$.
Hence, taking $l=1$,
\begin{align*}
 \delta_1=\epsilon_1\sqrt{ \frac{\epsilon_2 \mathcal V_-}{2\mathcal V_+}}, \quad 
 \delta_2=  \epsilon_1\sqrt{ \frac{\epsilon_2 \mathcal V_-}{2\mathcal V_\ast}},
\quad 
\mathcal T=\infty,
\end{align*}
yields
\begin{gather*}
\mathbb P\left(\sup_{t\geq t_\ast} |{\bf y}(t)|\geq \epsilon_1\right) \leq \epsilon_2.
\end{gather*}

Now, let $\mu^{(2p-n)/2}(t)\not\in L_1(t_0,\infty)$. In this case, $\zeta_{2p-n}(t)\to\infty$ as $t\to\infty$. Since $\mu(t)>0$ for all $t\geq \tau_0$, there exists $\mathcal T_\varepsilon>0$ such that 
\begin{gather*}
\mathcal V_1(t_\ast,\mathcal T_\varepsilon)=K_2 \varepsilon^{-2(1-l)}
\end{gather*}
for all $\varepsilon\in (0,\delta_2)$ and $l \in (0,1)$. It follows easily that $\mathcal T_\varepsilon\to \infty$ as $\varepsilon \to 0$.
By taking
\begin{gather*}
 \delta_1=\epsilon_1\sqrt{ \frac{\epsilon_2 \mathcal V_-}{2\mathcal V_+}}, \quad 
 \delta_2= \left( \frac{\epsilon_1^2\epsilon_2 \mathcal V_-}{2 K_2}\right)^{\frac{1}{2l}},
\quad 
\mathcal T=\mathcal T_\varepsilon, 
\end{gather*}
we obtain
\begin{gather*}
\mathbb P\left(\sup_{0\leq t-t_\ast\leq \mathcal T_\varepsilon} |{\bf y}(t)|\geq \epsilon_1\right) \leq \epsilon_2.
\end{gather*}

Thus, returning to the variables $\rho(t)$, $\psi(t)$, we get \eqref{rpst}.
 This completes the proof of Theorem~\ref{Th2}.

\section{Examples}\label{sex}

\subsection{Example 1}
Consider the system
\begin{gather}\label{Ex1}
\begin{split}
  dr = \left(\mu^2(t) a_1(r,\varphi,S(t)) +\mu^{2p}(t) c_1(r,\varphi,S(t))\right) dt + \varepsilon \mu^p(t) b_1(r,\varphi,S(t))  dw_1(t) \\
 d\varphi =\left(\nu(r)+\mu^2(t) a_2(r,\varphi,S(t)) +\mu^{2p}(t) c_2(r,\varphi,S(t))\right)dt+\varepsilon \mu^p(t) b_2(r,\varphi,S(t))  dw_1(t)
\end{split}
\end{gather}
where $\nu(r)\equiv 1-\vartheta r^2$, $S(t) \equiv s_0 t$, $\mu(t)\equiv t^{-1/2}\log t$,
\begin{align*}
\begin{pmatrix} a_1(r,\varphi,S)\\ a_2(r,\varphi,S)\end{pmatrix} & 
	\equiv (\mathcal Q(S) r \sin\varphi - \mathcal Z(S))\begin{pmatrix} \sin \varphi\\ r^{-1 }\cos \varphi \end{pmatrix}, \\
\begin{pmatrix} b_1(r,\varphi,S)\\ b_2(r,\varphi,S) \end{pmatrix} & 
	\equiv - \mathcal B(S) \begin{pmatrix} \sin \varphi\\ r^{-1 }\cos \varphi \end{pmatrix}, \\
\begin{pmatrix} c_1(r,\varphi,S)\\ c_2(r,\varphi,S) \end{pmatrix} & 
	\equiv \frac{(\varepsilon  \mathcal B (S))^2}{2r} \begin{pmatrix} \cos^2 \varphi\\ -r^{-1 }\sin 2 \varphi \end{pmatrix},
\end{align*}
$\mathcal Q(S)\equiv \mathcal Q_0+\mathcal Q_1\sin S$, $\mathcal Z(S)\equiv \mathcal Z_0+\mathcal Z_1\sin S$, $\mathcal B(S)\equiv \mathcal B_0+\mathcal B_1 \sin S$,
with constant parameters $s_0>0$, $\vartheta>0$, $\mathcal Q_k$, $\mathcal Z_k$ and $\mathcal B_k$. We see that system \eqref{Ex1} has the form \eqref{PS} with $\mathcal R=\vartheta^{-1/2}$, $m=4$, $\chi_m=0$,
\begin{align*}
{\bf a}(r,\varphi,S,t)\equiv &\mu^2(t)\begin{pmatrix}  a_1(r,\varphi,S)\\ a_2(r,\varphi,S) \end{pmatrix} + \mu^{2p}(t)\begin{pmatrix}  c_1(r,\varphi,S)\\ c_2(r,\varphi,S) \end{pmatrix}, \\ 
{\bf A}(r,\varphi,S,t) \equiv &\mu^p(t)\begin{pmatrix} b_1(r,\varphi,S) & 0 \\ b_2(r,\varphi,S) & 0 \end{pmatrix}.
\end{align*}
In the Cartesian coordinates $x_1=r\cos \varphi$, $x_2=-r\sin\varphi$ this system takes the form
\begin{align*}
&dx=(1-\vartheta (x_1^2+x_2^2)) x_2\, dt, \\ 
&dy=\left(-(1-\vartheta (x_1^2+x_2^2))x_1 + \mu^2(t)\mathcal G(x_1,x_2,S(t)) \right) dt+ \varepsilon \mu^{p}(t) \mathcal B(S(t))\, dw_1(t),
\end{align*}
where $ \mathcal G(x_1,x_2,S)\equiv \mathcal Z(S)+\mathcal Q(S) x_2$.

Let $s_0=1/2$. Then, there exist $\kappa=\varkappa=1$, $r_0=(2\vartheta)^{-1/2}$ such that the resonance condition \eqref{rc} holds with $\eta=-\sqrt{2\vartheta}<0$. 

1. Consider first the case $p=1$. It can easily be checked that the change of variables described in Theorem~\ref{Th1} with $N=4$ transforms the system to \eqref{rhopsi} with $n=2$, 
\begin{gather}\begin{split}\label{Ex1LO}
&\Lambda(\rho,\psi,t)\equiv \mu^{\frac 32}(t)\Lambda_3(\rho,\psi)+\mu^{2}(t)\Lambda_4(\rho,\psi)+o\left(\mu^{2}(t)\right), \\
&\Omega(\rho,\psi,t)\equiv \mu^{\frac{1}{2}}(t) \eta \rho - \mu(t) \vartheta \rho^2+\mu^{2}(t)\Omega_4(\rho,\psi)+o\left(\mu^{2}(t)\right)
\end{split}
\end{gather}
as $t\to\infty$ uniformly for all $|\rho|<\infty$, $(\psi,S)\in\mathbb R^2$, where
\begin{align*}
&\Lambda_3(\rho,\psi)\equiv \frac 12 \left(\frac{\mathcal Q_0}{\sqrt{2\vartheta}}-\mathcal Z_1 \cos\psi+\frac{\sqrt{2\vartheta} \varepsilon^2 }{8} \left(4\mathcal B_0^2+2\mathcal B_1^2-\mathcal B_1^2\cos 2\psi\right)\right), \\ 
&\Lambda_4(\rho,\psi)\equiv \frac{\rho}{2}\left(\mathcal Q_0-\frac{ \vartheta \varepsilon^2 }{4} \left(4\mathcal B_0^2+2\mathcal B_1^2-\mathcal B_1^2\cos 2\psi\right)\right),\\
&\Omega_4(\rho,\psi)\equiv \frac{1}{2}\left(\sqrt{2\vartheta} \mathcal Z_1 \sin\psi + \frac{ \vartheta \varepsilon^2  \mathcal B_1^2 \sin 2\psi}{2} \right).
\end{align*}

Let $\mathcal Z_1=0$, $\mathcal B_1\neq 0$ and 
\begin{gather}\label{exest1}
- \mathcal B_0^2-\frac{3\mathcal B_1^2}{4}<\frac{\mathcal Q_0}{\vartheta \varepsilon^2}<- \mathcal B_0^2-\frac{\mathcal B_1^2}{4}.
\end{gather}
Then, assumption \eqref{aszero} holds with
\begin{gather*}
\psi_0\in\left\{\pm \frac{\theta_0}{2} + \pi k, \quad k\in\mathbb Z\right\}, \quad \xi=\pm\frac{\sqrt{2\vartheta}\varepsilon^2 \mathcal B_1^2 }{8}\sin \theta_0, \\ \theta_0=\arccos \left(2+ \frac{4}{\mathcal B_1^2}\left(\mathcal B_0^2+\frac{\mathcal Q_0}{\vartheta \varepsilon^2}\right)\right).
\end{gather*}
Since $\eta<0$, we see that the equilibria $(0,-\theta_0/2+\pi k)$, corresponding to $\xi<0$, are unstable in the limiting system. Hence, the associated regime is not realized in the full system. Note that assumption \eqref{asst} holds with $h=4$ and $\tilde\gamma_h=\gamma_h=3\mathcal Q_0+\vartheta \varepsilon^2 (2\mathcal B_0^2+ \mathcal B_1^2)$. It follows from Lemma~\ref{Lem2} and Theorem~\ref{Th2} that if $\mathcal Q_0<-\vartheta \varepsilon^2 (2\mathcal B_0^2+ \mathcal B_1^2)/3$ and inequality \eqref{exest1} is satisfied, then a phase locking regime occurs with $r(t)\approx r_0$ and $\varphi(t)\approx S(t)+(\theta_0/2)({\hbox{\rm mod}} \,\pi)$ (see Fig.~\ref{FigEx11}). Moreover, it follows from \eqref{Tequat} that for all $l\in (0,1)$ this regime is stochastically stable on an asymptotically large time interval with $\mathcal T_\varepsilon=\varepsilon^{-2(1-l)}$. If \eqref{exest1} is not satisfied, then it follows from Lemma~\ref{Lem3} that the asymptotic regime with $r(t)\approx r_0$ and $\varphi(t)\approx S(t)$ does not arise.
\begin{figure}
\centering
{
   \includegraphics[width=0.4\linewidth]{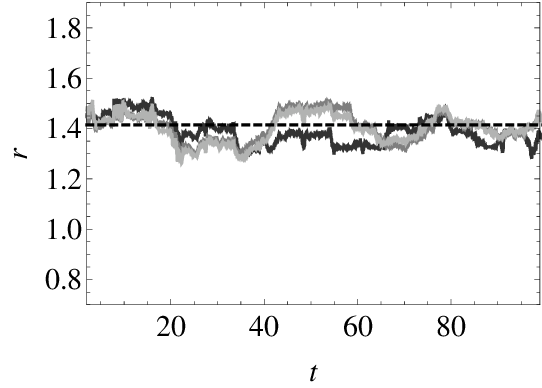}
}
\hspace{1ex}
{
   	\includegraphics[width=0.4\linewidth]{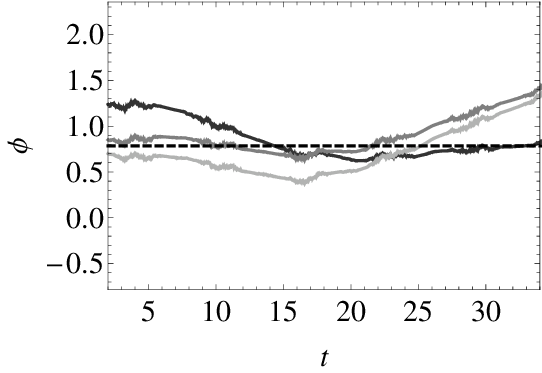}
}
\caption{\small The evolution of $r(t)$ and $\phi(t)\equiv \varphi(t)-S(t)$  for sample paths of solutions to system \eqref{Ex1} with $s_0=1/2$, $p=1$, $\vartheta=1/4$, $\mathcal Z_0=\mathcal Z_1=\mathcal B_0=\mathcal Q_1=0$, $\mathcal B_1=1$, $\mathcal  Q_0=-\varepsilon^2/8$, $\varepsilon=0.1$ and different values of initial data. The dashed curves correspond to $r(t)\equiv r_0$ and $\phi(t)\equiv \psi_0$, where $r_0=\sqrt 2$ and $\psi_0=\pi/4$.} \label{FigEx11}
\end{figure}

Let $\mathcal Z_1\neq 0$, $\mathcal B_1=0$ and 
\begin{gather}\label{exest2}
 \left|\frac{\mathcal Q_0+\vartheta \varepsilon^2\mathcal B_0^2}{\mathcal Z_1}\right|<\sqrt{2\vartheta}.
\end{gather}
Then, assumption \eqref{aszero} is satisfied with
\begin{gather*}
\psi_0\in\left\{\pm  \theta_0  + 2\pi k, \quad k\in\mathbb Z\right\}, \quad \xi=\pm\frac{\mathcal Z_1\sin \theta_0}{2}, \\ \theta_0=\arccos \left( \frac{\mathcal Q_0+\vartheta \varepsilon^2\mathcal B_0^2}{\mathcal Z_1 \sqrt{2\vartheta}}\right).
\end{gather*}
It can easily be checked that assumption \eqref{asst} holds with $h=4$ and $\gamma_h=\mathcal Q_0$. Hence, applying Theorem~\ref{Th2} shows that if $\pm \mathcal Z_1>0$, $\mathcal Q_0<0$ and inequality \eqref{exest2} is satisfied, then for all $l\in(0,1)$ a phase locking regime with $r(t)\approx r_0$ and $\varphi(t)\approx S(t)\pm\theta_0 ({\hbox{\rm mod}}\, 2\pi)$ occurs in system \eqref{Ex1} and is stochastically stable on asymptotically long time interval with $\mathcal T_\varepsilon=\varepsilon^{-2(1-l)}$ (see Fig.~\ref{FigEx12}).

\begin{figure}
\centering
{
   \includegraphics[width=0.4\linewidth]{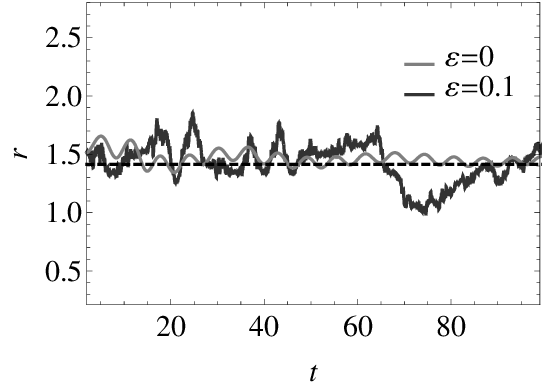}
}
\hspace{1ex}
{
   	\includegraphics[width=0.4\linewidth]{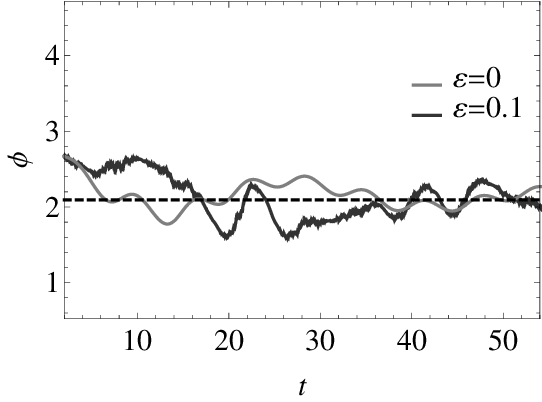}
}
\caption{\small The evolution of $r(t)$ and $\phi(t)\equiv \varphi(t)-S(t)$  for sample paths of solutions to system \eqref{Ex1} with $s_0=1/2$, $p=1$, $\vartheta=1/4$, $\mathcal Z_0=\mathcal B_1=\mathcal Q_1=0$, $\mathcal B_0=2$, $\mathcal Z_1=1/\sqrt{8}$, $\mathcal  Q_0=-0.126$ and different values of $\varepsilon$. The dashed curves correspond to $r(t)\equiv r_0$ and $\phi(t)\equiv \psi_0$, where $r_0=\sqrt 2$ and $\psi_0=2\pi/3$.} \label{FigEx12}
\end{figure}

2. Consider now the case $p=2$. The change of variables described in Theorem~\ref{Th1} with $N=4$ transforms the system to \eqref{rhopsi} with $n=2$, where $\Lambda(\rho,\psi,t)$ and $\Omega(\rho,\psi,t)$ satisfy \eqref{Ex1LO} with 
\begin{align*}
\Lambda_3(\rho,\psi)\equiv &\frac 12 \left(\frac{\mathcal Q_0}{\sqrt{2\vartheta}}-\mathcal Z_1 \cos\psi\right), \\
\Lambda_4(\rho,\psi)\equiv &\frac{\rho \mathcal Q_0}{2},\\
\Omega_4(\rho,\psi)\equiv &\frac{\sqrt{2\vartheta} \mathcal Z_1 \sin\psi}{2}.
\end{align*}

Let $\mathcal Z_1\neq 0$ and $|\mathcal Q_0/\mathcal Z_1|<\sqrt{2\vartheta}$. 
Then, assumption \eqref{aszero} is satisfied with
\begin{gather*}
\psi_0\in\left\{\pm  \theta_0  + 2\pi k, \quad k\in\mathbb Z\right\}, \quad \xi=\pm\frac{\mathcal Z_1\sin \theta_0}{2}, \\ \theta_0=\arccos \left( \frac{\mathcal Q_0}{\mathcal Z_1 \sqrt{2\vartheta}}\right),
\end{gather*}
and assumption \eqref{asst} holds with $h=4$ and $\tilde \gamma_h=\gamma_h=\mathcal Q_0$. Note that if $\pm \mathcal Z_1<0$, then the equilibria $(0,\pm\theta_0+2\pi k)$ are unstable in the corresponding limiting system. By applying Theorem~\ref{Th2}, we see that if $\pm \mathcal Z_1>0$ and $-|\mathcal Z_1|\sqrt{2\vartheta}<\mathcal Q_0<0$, then for all $l\in (0,1)$ a phase locking regime with $r(t)\approx r_0$ and $\varphi(t)\approx S(t)\pm\theta_0 ({\hbox{\rm mod}}\, 2\pi)$ persists in the perturbed stochastic system \eqref{Ex1} at least over asymptotically long time interval with $\mathcal T_\varepsilon=\mathcal O(\varepsilon^{-4(1-l)})$ as $\varepsilon \to 0$. If $|\mathcal Q_0/\mathcal Z_1|>\sqrt{2\vartheta}$, then it follows from Lemma~\ref{Lem3} that the asymptotic regime with $r(t)\approx r_0$ and $\varphi(t)\approx S(t)$ does not occur  (see Fig.~\ref{FigEx13}).

\begin{figure}
\centering
{
    \includegraphics[width=0.4\linewidth]{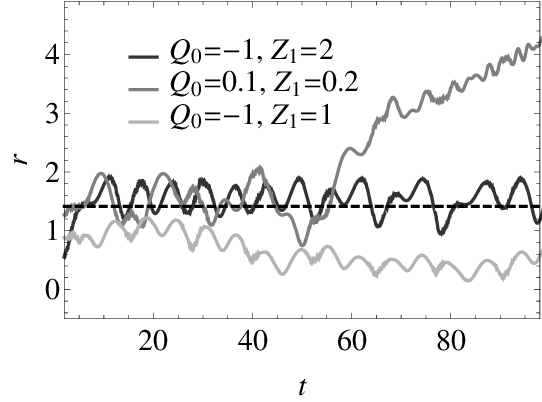}
}
\hspace{1ex}
{
    	\includegraphics[width=0.4\linewidth]{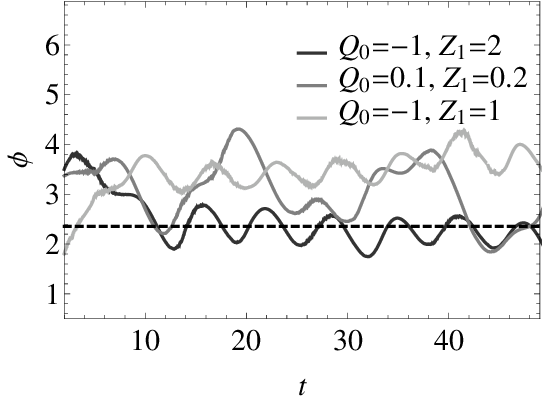}
}
\caption{\small The evolution of $r(t)$ and $\phi(t)\equiv \varphi(t)-S(t)$  for sample paths of solutions to system \eqref{Ex1} with $s_0=1/2$, $p=2$, $\vartheta=1/4$, $\mathcal Z_0=\mathcal  Q_1=0$, $\mathcal B_0=\mathcal B_1=1$, $\varepsilon=0.1$ and different values of the parameters $\mathcal Q_0$ and $\mathcal Z_1$. The dashed curves correspond to $r(t)\equiv r_0$ and $\phi(t)\equiv \psi_0$, where $r_0=\sqrt 2$ and $\psi_0=3\pi/4$.} \label{FigEx13}
\end{figure}

\subsection{Example 2} 
Consider again equation \eqref{Ex0}. It follows from Section~\ref{sec1} that this system correspond to \eqref{PS} with $\mu(t)=t^{-1/4}$, $s_0=3/2$, and $\nu(r)$, ${\bf a}(r,\varphi,S,t)=(a_1(r,\varphi,S,t),a_2(r,\varphi,S,t))^T$ and ${\bf A}(r,\varphi,S,t)=\{\alpha_{i,j}(r,\varphi,S,t)\}_{2\times 2}$ defined by \eqref{omegaeq} and \eqref{aAEx0}. In this case, the condition \eqref{mucond} is satisfied with $m=4$ and $\chi_m=-1/4$. Note that $0<\nu(r)<1$ for all $0<|r|<(2\vartheta)^{-1/2}$ and $\nu(r)=1-3 \vartheta r^2/8-35\vartheta^2 r^4/256+\mathcal O(\vartheta^4)$ as $\vartheta\to 0$. Hence, there exist $\kappa, \varkappa\in\mathbb Z_+$ and $0<r_0<(2\vartheta)^{-1/2}$ such that the condition \eqref{rc} holds with $\eta<0$. Note that the root of the equation $\nu(r_0)=\kappa s_0/\varkappa$ can be found numerically. In particular, if $\vartheta=2^{-5}$, $\kappa=1$ and $\varkappa=2$, then $r_0\approx 3.6$ (see Fig.~\ref{fignu}). 

\begin{figure}
\centering
{
   \includegraphics[width=0.4\linewidth]{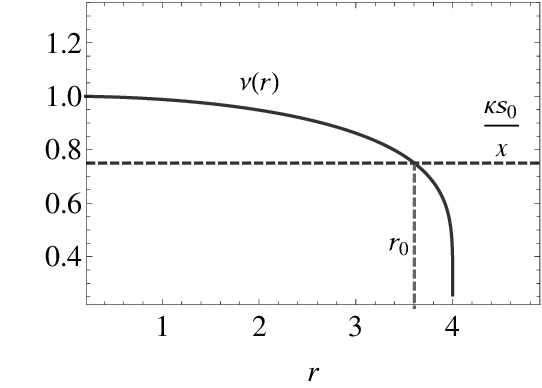}
}
\caption{\small Finding $r_0$, when $\vartheta=2^{-5}$, $\kappa=1$ and $\varkappa=2$.} \label{fignu}
\end{figure} 

1. Consider first the case $n=2$, $p=1$. Let $\kappa=1$ and $\varkappa=2$. Then, the transformation described in Theorem~\ref{Th1} with $N=4$ reduces system to \eqref{rhopsi} with 
\begin{gather}\begin{split}\label{Ex2LO}
&\Lambda(\rho,\psi,t)\equiv t^{-\frac{3}{8}}\Lambda_3(\rho,\psi)+t^{-\frac{1}{2}}\Lambda_4(\rho,\psi)+\mathcal O(t^{-\frac{5}{8}}), \\
&\Omega(\rho,\psi,t)\equiv t^{-\frac{1}{8}}\eta \rho +\sum_{k=2}^4 t^{-\frac{k}{8}}\Omega_k(\rho,\psi)+\mathcal O(t^{-\frac{5}{8}})
\end{split}
\end{gather}
as $t\to\infty$ uniformly for all $|\rho|<\infty$, $(\psi,S)\in\mathbb R^2$, where
\begin{align*}
&\Lambda_3(\rho,\psi)= \frac{r_0}{4} \left(2\mathcal Q_0+\frac{\varepsilon^2 \mathcal B_0^2}{r_0^2}-\mathcal Z_1 \cos(2\psi+\theta_1)\right)+\mathcal O(\vartheta), \\ 
&\Lambda_4(\rho,\psi)= \frac{\rho}{4}\left(2\mathcal Q_0-\frac{\varepsilon^2 \mathcal B_0^2}{r_0^2}-\mathcal Z_1 \cos(2\psi+\theta_1)\right)+\mathcal O(\vartheta),\\
& \Omega_2(\rho,\psi)\equiv \frac{\partial_r^2 \nu(r_0)\rho^2}{2},\\
& \Omega_3(\rho,\psi)\equiv \frac{\partial_r^3 \nu(r_0)\rho^3}{6},\\
&\Omega_4(\rho,\psi)= \frac{\partial_r^4 \nu(r_0)\rho^4}{24}+ \frac{3}{16}\left(-2\mathcal P_0+\mathcal Z_1 \sin(2\psi+\theta_1)  \right)+\mathcal O(\vartheta)
\end{align*}
 as $\vartheta\to 0$, $\mathcal Z_1=\sqrt{\mathcal P_1^2+\mathcal Q_1^2}$, $\theta_1=\arccos(\mathcal P_1/\mathcal Z_1)$. 
It can easily be checked that if $\mathcal Z_1\neq 0$ and the inequalities 
\begin{gather}\label{ex2est1}
- \frac{1}{2}\left(\mathcal Z_1+\frac{\varepsilon^2\mathcal B_0^2}{ r_0^2}\right)<\mathcal Q_0< \frac{1}{2}\left(\mathcal Z_1-\frac{\varepsilon^2\mathcal B_0^2}{ r_0^2}\right)
\end{gather}
are satisfied, then assumption \eqref{aszero} holds with
\begin{gather*}
 \psi_0\in\left\{\frac{\pm \theta_0-\theta_1}{2} + \pi k+\mathcal O(\vartheta), \quad k\in\mathbb Z\right\}, \quad
\xi=\pm\frac{r_0 \mathcal Z_1}{2}\sin \theta_0+\mathcal O(\vartheta),\\
  \theta_0=\arccos \left(\frac{1}{\mathcal Z_1}\left(2\mathcal Q_0+ \frac{\varepsilon^2 \mathcal B_0^2}{r_0^2}\right)\right).
\end{gather*}
From the stability analysis of the limiting system it follows that if $\xi<0$, then the regimes corresponding to the equilibria $(0,-(\theta_0+\theta_1)/2+\pi k)$ are not realized. We see that assumption \eqref{asst} holds with $h=4$ and 
\begin{gather*}
\gamma_h=\frac{3}{4}\left(\mathcal Q_0-\frac{\varepsilon^2 \mathcal B_0^2}{6 r_0^2}\right)+\mathcal O(\vartheta), \quad \vartheta\to 0.
\end{gather*}
From \eqref{zetah} it follows  that $\tilde \gamma_h=\gamma_h$. Thus, by Lemma~\ref{Lem2} and Theorem~\ref{Th2}, we see that if 
\begin{gather}\label{Qoineq}
\mathcal Q_0<\frac{\varepsilon^2 \mathcal B_0^2}{6 r_0^2}
\end{gather}
and \eqref{ex2est1} holds, then for small enough $\vartheta\ll \varepsilon^2$ a phase locking regime occurs with $r(t)\approx r_0$ and $\varphi(t)\approx S(t)/2+((\theta_0-\theta_1)/2)({\hbox{\rm mod}} \,\pi)$  (see Fig.~\ref{FigEx2}). Moreover, for all $l\in (0,1)$ this regime is stochastically stable on an asymptotically large time interval with $\mathcal T_\varepsilon=\varepsilon^{-2(1-l)}$.
Finally, from Lemma~\ref{Lem3} it follows that if \eqref{ex2est1} is not satisfied, then the asymptotic regime with $r(t)\approx r_0$ and $\varphi(t)\approx S(t)/2$ does not arise.

Note that if $\varepsilon = 0$, then the noise terms disappear in the perturbed system \eqref{Ex0}. It follows from \cite{OS25DCDS} that in this case a stable phase locking regime occurs if $\mathcal Z_1\neq 0$ and $-\mathcal Z_1/2<\mathcal Q_0<0$. Comparing these inequalities with \eqref{ex2est1} and \eqref{Qoineq}, we see that stochastic perturbations may expand the stability domain in the parameter space. However, in this case the length of the stability interval is finite, but asymptotically large. 

\begin{figure}
\centering
{
    \includegraphics[width=0.4\linewidth]{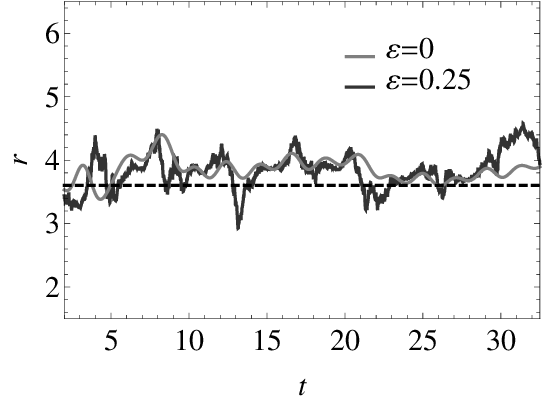}
}
\hspace{1ex}
{
    	\includegraphics[width=0.4\linewidth]{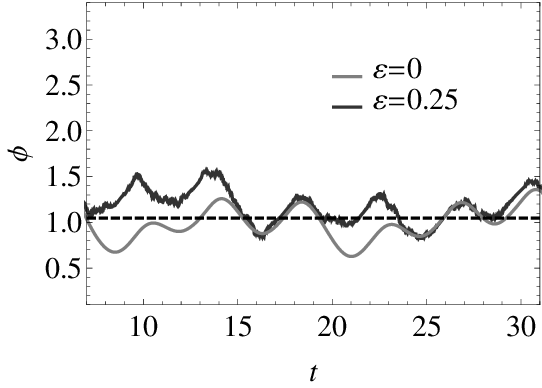}
}
\caption{\small The evolution of $r(t)\equiv \sqrt{2U(x(t))+y^2(t)}$ and $\phi(t)\equiv \varphi(t)-S(t)/2$  for sample paths of solutions to system \eqref{Ex0} with $s_0=3/2$, $n=2$, $p=1$, $\vartheta=2^{-5}$, $\mathcal P_0=\mathcal  Q_1=0$, $\mathcal P_1=1$, $\mathcal Q_0=-0.25$, $\mathcal B_0=3.6$, $\mathcal B_1=0$. The dashed curves correspond to $r(t)\equiv r_0$ and $\phi(t)\equiv \psi_0$, where $r_0\approx 3.6$ and $\psi_0 \approx 1.01$.} \label{FigEx2}
\end{figure}

2. Consider now the case $n=p=2$. If $\kappa=1$ and $\varkappa=2$, then the averaged system takes the form \eqref{rhopsi} with $\Lambda(\rho,\psi,t)$ and $\Omega(\rho,\psi,t)$ satisfying \eqref{Ex2LO}, where
\begin{align*}
&\Lambda_3(\rho,\psi)= \frac{r_0}{4} \left(2\mathcal Q_0-\mathcal Z_1 \cos(2\psi+\theta_1)\right)+\mathcal O(\vartheta), \\ 
&\Lambda_4(\rho,\psi)= \frac{\rho}{4}\left(2\mathcal Q_0-\mathcal Z_1 \cos(2\psi+\theta_1)\right)+\mathcal O(\vartheta),\\
& \Omega_2(\rho,\psi)\equiv \frac{\partial_r^2 \nu(r_0)\rho^2}{2},\\
& \Omega_3(\rho,\psi)\equiv \frac{\partial_r^3 \nu(r_0)\rho^3}{6},\\
&\Omega_4(\rho,\psi)= \frac{\partial_r^4 \nu(r_0)\rho^4}{24}+ \frac{3}{16}\left(-2\mathcal P_0+\mathcal Z_1 \sin(2\psi+\theta_1)  \right)+\mathcal O(\vartheta)
\end{align*}
as $\vartheta\to 0$, $\mathcal Z_1=\sqrt{\mathcal P_1^2+\mathcal Q_1^2}$ and $\theta_1=\arccos(\mathcal P_1/\mathcal Z_1)$. 
If $\mathcal Z_1\neq 0$ and $|\mathcal Q_0|< \mathcal Z_1/2$, then assumption \eqref{aszero} holds with
\begin{gather*}
 \psi_0\in\left\{\frac{\pm \theta_0-\theta_1}{2} + \pi k+\mathcal O(\vartheta), \quad k\in\mathbb Z\right\}, \quad
\xi=\pm\frac{r_0 \mathcal Z_1}{2}\sin \theta_0+\mathcal O(\vartheta),\\
  \theta_0=\arccos \frac{2\mathcal Q_0}{\mathcal Z_1}.
\end{gather*}
Since $\eta<0$, we see that the equilibria $(0,-(\theta_0+\theta_1)/2+\pi k)$, corresponding to $\xi<0$, are unstable in the limiting system. It can easily be checked that assumption \eqref{asst} holds with $h=4$ and $\tilde \gamma_h=\gamma_h=3\mathcal Q_0/4$. Thus, by Lemma~\ref{Lem2} and Theorem~\ref{Th2}, we see that if $\mathcal Q_0<0$, then for all $l\in(0,1)$ and for small enough $\vartheta\ll 1$ a phase locking regime occurs with $r(t)\approx r_0$ and $\varphi(t)\approx S(t)/2+((\theta_0-\theta_1)/2)({\hbox{\rm mod}} \,\pi)$, and is stochastically stable on asymptotically long time interval with $\mathcal T_\varepsilon=\mathcal O(\varepsilon^{-8(1-l)/3})$ as $\varepsilon \to 0$  (see Fig.~\ref{FigEx22}). Note that, unlike the previous case, stochastic perturbations in this case do not lead to a change in the stability boundaries.

\begin{figure}
\centering
{
   \includegraphics[width=0.4\linewidth]{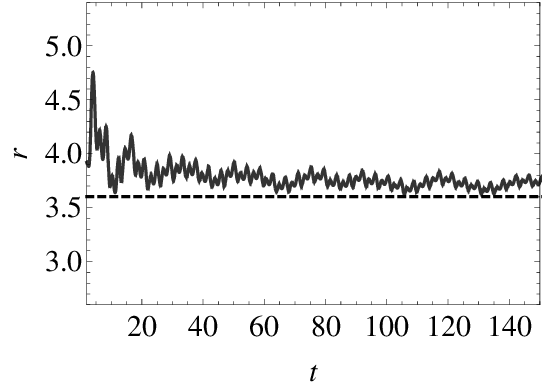}
}
\hspace{1ex}
{
    	\includegraphics[width=0.4\linewidth]{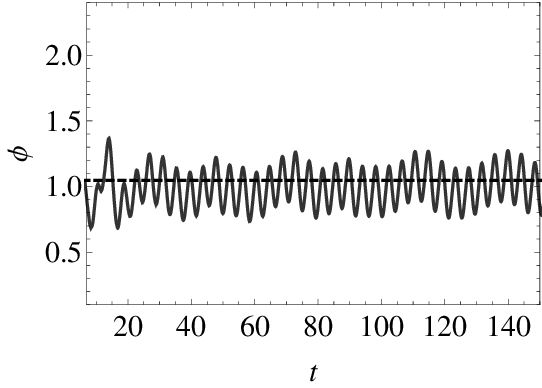}
}
\caption{\small The evolution of $r(t)\equiv \sqrt{2U(x(t))+y^2(t)}$ and $\phi(t)\equiv \varphi(t)-S(t)/2$  for sample paths of solutions to system \eqref{Ex0} with $s_0=3/2$, $n= p=2$, $\vartheta=2^{-5}$, $\varepsilon=0.125$, $\mathcal P_0=\mathcal  Q_1=0$, $\mathcal P_1=1$, $\mathcal Q_0=-0.25$, $\mathcal B_0=1$, $\mathcal B_1=0$. The dashed curves correspond to $r(t)\equiv r_0$ and $\phi(t)\equiv \psi_0$, where $r_0\approx 3.6$ and $\psi_0 =\pi/3$.} \label{FigEx22}
\end{figure}

\section{Conclusion}

Thus, the combined influence of stochastic perturbations and decaying driving with asymptotically constant frequency on the non-isochronous systems far from the equilibrium have been investigated. In particular, the persistence of phase locking phenomena and the appearance of solutions with resonant amplitude values in the presence of noise have been studied. 

By averaging the drift terms of equations with respect to the phase of perturbations, we have derived a model truncated system \eqref{trsys}, describing possible asymptotic regimes --- phase locking and phase drifting. The resonant solutions arise in the phase locking mode. We have described the conditions that guarantee the persistent of such solutions in the full stochastic system on infinite or asymptotically large time intervals depending on the intensity of the noise. The examples presented illustrate how the developed theory can be applied. In particular, the resonant capture in the stochastically perturbed Duffing oscillator have been discussed. We have shown that in some cases, stochastic perturbations may expand the stability domain in the parameter space and can be used for stabilization.

Note that the proposed study have been focused on perturbations with decaying intensity characterized by the function $\mu(t)$ that satisfies specific conditions \eqref{mucond}. We have shown that power and power-logarithmic functions are suitable. It is easy to check that, for instance, decaying exponential functions do not satisfy these conditions. In such cases, the developed theory cannot be applied directly. Note also that additional difficulties may arise due to internal resonances and the problem of small denominators when considering perturbations of multidimensional systems. These cases deserve special attention and will be considered elsewhere.

\section*{Acknowledgments}
The work is supported by the Russian Science Foundation (project No. 19-71-30002).

\end{document}

}
\end{document}